\documentclass[a4paper,10pt]{amsart}
\usepackage{amsmath,amsthm,latexsym,amscd,amssymb}
\usepackage[all]{xy}
\usepackage[hypertex,backref]{hyperref}
\usepackage{epsfig,color}
\numberwithin{equation}{section}

\newtheorem{prop}{Proposition}[section]
\newtheorem{lem}[prop]{Lemma}
\newtheorem{sublem}[prop]{Sublemma}
\newtheorem{ass}[prop]{Assumption}
\newtheorem{ddd}[prop]{Definition}
\newtheorem{theorem}[prop]{Theorem}
\newtheorem{cor}[prop]{Corollary}

\newcommand{\Cl}{{\rm C}}

\newcommand{\gradu}{\mathop{\bf z}\nolimits}

\newcommand{\sym}{\mathbf S}

\newcommand{\dom}{\mathop{\rm dom}}

\newcommand{\Pj}{{\mathcal P}}

\newcommand{\E}{{\mathcal E}}

\newcommand{\F}{{\mathcal F}}

\newcommand{\D}{{\mathcal D}}
\newcommand{\Dk}{{\mathcal D}^{prod}}

\newcommand{\Ca}{{\mathcal C}}
\newcommand{\A}{{\mathcal A}}
\newcommand{\B}{{\mathcal B}}
\newcommand{\cH}{{\mathcal H}}
\newcommand{\inv}{\mathop{\mathcal I}\nolimits}

\newcommand{\C}{C^{\infty}}

\newcommand{\ra}{\partial}
\newcommand{\lr}{\longrightarrow}

\newcommand{\ten}{\otimes}

\newcommand{\ve}{\varepsilon}
\newcommand{\ov}{\overline}

\newcommand{\dira}{\partial\!\!\!/}
\newcommand{\dirac}{\partial\!\!\!/}

\DeclareMathOperator{\sign}{sign}

\DeclareMathOperator{\Ker}{Ker}
\DeclareMathOperator{\Ima}{Im}

\def\bbbr{{\rm I\!R}} 

\def\bbbc{{\rm I\!C}}

\def\bbbq{{\mathchoice {\setbox0=\hbox{$\displaystyle\rm Q$}\hbox{\raise
0.15\ht0\hbox to0pt{\kern0.4\wd0\vrule height0.8\ht0\hss}\box0}}
{\setbox0=\hbox{$\textstyle\rm Q$}\hbox{\raise
0.15\ht0\hbox to0pt{\kern0.4\wd0\vrule height0.8\ht0\hss}\box0}}
{\setbox0=\hbox{$\scriptstyle\rm Q$}\hbox{\raise
0.15\ht0\hbox to0pt{\kern0.4\wd0\vrule height0.7\ht0\hss}\box0}}{\setbox0=\hbox{$\scriptscriptstyle\rm Q$}\hbox{\raise
0.15\ht0\hbox to0pt{\kern0.4\wd0\vrule height0.7\ht0\hss}\box0}}}}

\def\bbbz{{\mathchoice {\hbox{$\sf\textstyle Z\kern-0.4em Z$}}
{\hbox{$\sf\textstyle Z\kern-0.4em Z$}}
{\hbox{$\sf\scriptstyle Z\kern-0.3em Z$}}
{\hbox{$\sf\scriptscriptstyle Z\kern-0.2em Z$}}}}

\def\bbbc{{\mathchoice {\setbox0=\hbox{$\displaystyle\rm C$}\hbox{\hbox
to0pt{\kern0.4\wd0\vrule height0.9\ht0\hss}\box0}}
{\setbox0=\hbox{$\textstyle\rm C$}\hbox{\hbox
to0pt{\kern0.4\wd0\vrule height0.9\ht0\hss}\box0}}
{\setbox0=\hbox{$\scriptstyle\rm C$}\hbox{\hbox
to0pt{\kern0.4\wd0\vrule height0.9\ht0\hss}\box0}}
{\setbox0=\hbox{$\scriptscriptstyle\rm C$}\hbox{\hbox
to0pt{\kern0.4\wd0\vrule height0.9\ht0\hss}\box0}}}}

\title{Product formula for Atiyah-Patodi-Singer index classes and higher signatures}
\author{Charlotte Wahl}

\begin{document}

\begin{abstract}
We define generalized Atiyah-Patodi-Singer boundary conditions of product type for Dirac operators associated to $C^*$-vector bundles on the product of a compact manifold with boundary and a closed manifold. We prove a product formula for the $K$-theoretic index classes, which we use to generalize the product formula for the topological signature to higher signatures.
\end{abstract} 

\maketitle
\markright{PRODUCT FORMULA FOR APS-CLASSES}

\section{Introduction}
It is an elementary fact from algebraic topology that the topological signature fulfills
$$\sign(M) \cdot \sign(N)=\sign(M\times N) \ ,$$
if $M$ is an oriented compact manifold with boundary and $N$ is an oriented closed manifold. In this paper we prove a similar product formula for higher signatures -- more generally: for the signature classes of the signature operator twisted by a flat $C^*$-vector bundle. (In the higher case this bundle is the Mishenko-Fomenko bundle.)

In the closed case the signature class equals the $K$-theoretic index of the signature operator. There are several definitions of a higher signature class for a manifold with boundary, which conjecturally give the same class (see \cite[\S 13 I]{lp7}): Two analytic ones (whose Chern characters agree), see \cite{llp}, and a topological definition based on $L$-theory \cite{llk}. We refer to the survey \cite{lp7} for a historical account. The basis for our considerations is the definition of the signature class as the index of the signature operator with generalized Atiyah-Patodi-Singer boundary conditions given by a symmetric spectral section \cite{lp4}\cite{lp6}. The class is well-defined only under certain homological conditions. We prove the following generalization of the above formula: Let $\A, \B$ be unital $C^*$-algebras. If $\F_M$ resp. $\F_N$ is a flat unitary $\A$- resp. $\B$-vector bundle on and even-dimensional manifold $M$ resp. $N$, then 
$$\sigma(M,\F_M) \ten \sigma(N,\F_N)=\sigma(M\times N,\F_M \boxtimes \F_N) \in K_0(\A\ten \B) \ ,$$
if both sides are defined. Here $\sigma(M,\F_M) \in K_0(\A)$ resp. $\sigma(N,\F_N) \in K_0(\B)$ are the signature classes. If $M$ or $N$ is odd-dimensional, there is a similar formula, however the signature depends then on the additional choice of a Lagrangian. The actual result we prove is slightly more general such that it applies to higher signatures (see \S \ref{highsig}).

The proof of the signature formula builds on a product formula for Atiyah-Patodi-Singer index classes (Theorem \ref{prod}), which is the main result of the first part of this paper (\S \ref{proddir}). We use a class of boundary conditions of Atiyah-Patodi-Singer type that generalizes the boundary conditions introduced in \cite{mp1}\cite{mp2} for families and adapted in \cite{lp2}\cite{lp6} to higher index theory. In this class we can associate to any boundary condition for a Dirac operator on $M$ a canonical boundary condition for a suitable product Dirac operator on the product $M \times N$. The proof of the product formula is based on $KK$-theoretical methods, in particular the relative index theorem \cite{bu}. It carries over to family index theory, where a product formula might also be of interest. A special case is the equality between the Dirac operator and its Dirac suspension, which was defined and established in \cite[\S 5]{mp2} in the family case and adapted to the noncommutative context in \cite[\S 3]{lp6}. (Note the following subtlety: In \cite{mp2}\cite{lp6} odd index classes were defined in terms of a suspension map originally due to Atiyah and Singer. Here we use a $KK$-theoretic approach, which is makes calculations more straightforward and allows to treat the even and odd case on an equal footing. The index classes defined by both approaches agree, see \cite[\S 9]{wa1}.) 

The product formula for Atiyah-Patodi-Singer classes has applications to the study of concordance classes of metrics of positive scalar curvature: Stolz defined bordism groups $R_n(\pi)$ for a finitely presented group $\pi$ (in fact, more generally for so-called supergroups) \cite{sto}\cite[\S 5]{rs}. These groups consist of equivalence classes of $n$-dimensional spin manifolds with boundary that are endowed with a reference map to $B\pi$ and with a metric of positive scalar curvature on the boundary. Taking the index of the Dirac operator twisted by the Mishenko-Fomenko bundle associated to the maximal group $C^*$-algebra yields a homomorphism $R_n(\pi) \to K_n(C_{max}^*\pi)$ (see \cite[\S 1.4]{bu}, with the real reduced $C^*$-algebra used there replaced by $C_{max}^*\pi$).
For finitely presented groups $\pi_1,\pi_2$ the Cartesian product induces a product $R_n(\pi_1) \times\Omega_m^{spin}(B\pi_2) \to R_{n+m}(\pi_1\times \pi_2)$. There is also an index map $\Omega_m^{spin}(B\pi_2) \to K_m(C^*_{max}\pi_2)$. By the product formula for Atiyah-Patodi-Singer classes these maps fit into a commuting diagram 
$$\xymatrix{
R_{n}(\pi_1) \times \Omega_m^{spin}(B\pi_2) \ar[r]\ar[d]& R_{n+m}(\pi_1 \times \pi_2) \ar[d] \\
K_n(C^*_{max}\pi_1) \ten K_m(C^*_{max} \pi_2) \ar[r]^{\ten} &K_{n+m}(C^*_{max}(\pi_1 \times \pi_2))\ .}$$
This can be applied to study the behavior of the concordance classes under Cartesian product, see \cite[Remark 0.7]{we} for related questions. We expect that our methods also work in $KO$-theory, which should be used here: The index maps in the diagram factor through $KO$-theory of the corresponding real maximal $C^*$-algebras. A special case of the analogue of the above diagram in $KO$-theory is the fact that the homomorphism $\varinjlim R_{n+8j}(\pi) \to KO_n(C^*_{\bbbr,max}\pi)$  is well-defined: The limit is induced by taking the product with a particular closed 8-dimensional manifold (the Bott manifold) \cite{sto}\cite[\S 5]{rs}. A more general diagram is given in the preprint \cite{sto}, which was never published. Also for the above diagram  (resp. its analogue in $KO$-theory) there seems to be no published proof.

A novelty used in the proof of the product formula for signature classes is a generalization of the definition of symmetric boundary conditions for the signature operator. Symmetric spectral sections, as introduced in \cite{lp4}\cite{lp6}, are symmetric with respect to a particular involution. The class of boundary conditions defined by symmetric spectral sections is not closed under taking products. We consider more general involutions and study the dependence of the involution. The results allow us to derive the product formula for the signature classes from the product formula for Atiyah-Patodi-Singer classes.

It would be interesting to have a similar product formula established for the topologically defined higher signatures. In general, the main advantage of the $K$-theoretical approach is that it also works for foliations, as noted in Remark 2 at the end of \cite{lp6}.

The methods of the present paper together with the product formula for $\eta$-forms proven in \cite{waaps} also lead to a product formula for the analytic higher $\rho$-invariants for the signature operator. (Details will be given elsewhere.) These were defined in \cite{waaps} motivated by a suggestion in \cite{lohigheta}. An alternative definition based on a different regularization can be given using the higher $\eta$-forms for the signature operator introduced in \cite{llp}. Topological higher $\rho$-invariants were previously introduced in \cite{we}. There Cartesian products were the motivating examples, and a product formula was mentioned. A connection to the analytic definition has not yet been established.

\subsection*{Conventions}

If not specified, a tensor product between $C^*$-algebras is understood as the spatial (=minimal) $C^*$-algebraic tensor product, and a tensor product between Hilbert $C^*$-modules is the exterior Hilbert $C^*$-module tensor product. In the few remaining cases the tensor product is assumed to be algebraic. A tensor product of graded spaces is graded. However, for operators we fix the following convention: If $A$ resp. $B$ are operators on graded vector spaces $H_1$ resp. $H_2$, then $A \ten B$ is the operator on $H_1 \ten H_2$ defined by using the {\it ungraded} tensor product, hence neglecting the grading. In contrast the operator $AB$ on $H_1 \ten H_2$ is defined via the graded tensor product as usual. Thus $AB=A \ten B^+ + A \gradu \ten B^-$, where $\gradu$ is the grading operator on $H_1$ and $B=B^+ +B^-$ with $B^{\pm}$ even resp. odd. In this spirit we usually omit tensor products when dealing with operators and write $A$ for $A \ten 1$ resp. $B$ for $1 \ten B^+ + \gradu \ten B^-$. We also usually omit the tensor product when dealing with morphisms between different spaces. In a graded context we tacitly endow ungraded spaces with the trivial $\bbbz/2$-grading (for which all elements are positive).

In order to avoid confusion we add indices to geometric operators as the de Rham operator. We will omit them sometimes when confusion seems unlikely.

\section{Product formula for Dirac classes}
\label{proddir}

We assume throughout the paper that $\A, \B$ are unital $C^*$-algebras. 

Let $M$ be an oriented Riemannian manifold with boundary $\ra M$ and product structure near the boundary. Denote by $M_{cyl}$ the corresponding manifold with cylindric end $Z_r \subset M_{cyl}$. That is, we assume that there is $\ve>0$ and an isometry $e:Z_r \cong (-\ve,\infty) \times \ra M$ such that $M_{cyl}\setminus e^{-1}((0,\infty) \times \ra M) =M$. The coordinate defined by the composition of $e$ with the projection onto $(-\ve,\infty)$ is denoted by $x_1$. We define $Z=\bbbr \times \ra M$. We set $U_{\ve}=e^{-1}((-\ve,0]\times \ra M) \subset M$ and denote by $p:U_{\ve}\to \ra M$ the composition of $e$ with the projection onto $\ra M$. The projection $Z \to \ra M$ will be denoted by $p$ as well.

Dirac operators over $C^*$-algebras are by now well-studied. It turns out that much of the classical theory carries over, see for example \cite{st} \cite{s} for relevant background material. 

Let $\E$ be a hermitian $\A$-vector bundle on $M$ (the scalar product on the fibers is assumed to be $\A$-valued). Then $\E$ is called a Dirac $\A$-bundle if the following conditions are fulfilled:
\begin{enumerate}
\item The bundle $\E$ is a Clifford module. This means that there is a left action of the Clifford bundle $\Cl(T^*M)$ on $\E$ commuting with the right action of $\A$ such that the $c(v)$ is a skewadjoint endomorphism on $\E$ for any $v \in T^*M$. If $M$ is even-dimensional, then $\E$ is assumed to be $\bbbz/2$-graded and $c(v)$ is assumed to be odd for any $v \in T^*M$. 
\item Furthermore $\E$ is endowed with a connection $\nabla^{\E}$ compatible with the hermitian product and fulfilling $c(\nabla^Mv)=[\nabla^{\E},c(v)]$. Here $\nabla^M$ is the Levi-Civit\`a connection. 
\end{enumerate}

Let $\E_M$ be a Dirac $\A$-bundle on $M$ and assume that $\E_M|_{U_{\ve}}=p^*(\E_M|_{\ra M})$ as (graded, if $M$ is even-dimensional) hermitian $\A$-vector bundle. Furthermore the connection on $\E_M|_{U_{\ve}}$ is assumed to be of product type. Let $\dirac_{M}:=c \circ \nabla^{\E_M}$ be the associated Dirac operator. 

The bundle $\E_M$ is $\bbbz/2$-graded if $M$ is even-dimensional. The grading operator is denoted by $\gradu_M$. We write $\E_{\ra M}:=\E_M^+|_{\ra M}$ if $M$ is even-dimensional and $\E_{\ra M}=\E|_{\ra M}$ if $M$ is odd-dimensional. 

The induced Clifford module structure on $\E_{\ra M}$ is given by $c_{\ra M}(v):=c_M(dx_1)c_M(v)$ for $v \in T^*\ra M \subset T^*M$ (the inclusion being defined via the metric). We denote the Dirac operator associated to $\E_{\ra M}$ by $\dira_{\ra M}$. If $M$ is odd-dimensional, the Dirac bundle $\E_{\ra M}$ is $\bbbz/2$-graded with grading operator $\gradu_{\ra M}:=ic_M(dx_1)$ and on $U_{\ve}$
\begin{align}
\label{boundopodd}
\dira_M= c_M(dx_1)(\ra_1 - \dira_{\ra M}) \ .
\end{align}
If $M$ is even-dimensional, we identify $\E^+|_{U_{\ve}}$ with $\E^-|_{U_{\ve}}$ via $ic(dx_1)$ and thus obtain an isomorphism
$$\E|_{U_{\ve}}\cong (\bbbc^+ \oplus \bbbc^-)  \ten (p^*\E_{\ra M}) \ .$$ Here $\bbbc^{\pm}$ denotes $\bbbc$ with grading induced by the grading operator $\pm 1$. 
On $U_{\ve}$
\begin{align}
\label{boundopev}
\dira_M&= c_M(dx_1)(\ra_1 -\gradu_M \dira_{\ra M}) \ .
\end{align}
Given $\dira_M$, the operator $\dira_{\ra M}$ is uniquely determined by these formulas and is called the boundary operator induced by $\dira_M$. In the following the boundary operator of a Dirac operator $\dira$ will sometimes be denoted by $B(\dira)$.

Write $\D_{\ra M}$ for the closure of $\dira_{\ra M}:\C(\ra M,\E_{\ra M}) \to L^2(\ra M,\E_{\ra M})$.

Now we introduce the boundary conditions:

Assume first that $M$ is even-dimensional. Then a selfadjoint operator $A \in B(L^2(\ra M,\E_{\ra M}))$ such that $\D_{\ra M}+A$ has a bounded inverse is called a trivializing operator for $\D_{\ra M}$ on $L^2(\ra M,\E_{\ra M})$. 

Define $\D_M(A)^+$ as the closure of 
$$\dirac_M^+:\{f \in \C(M,\E^+)~|~ 1_{\ge 0}(\D_{\ra M}+A)(f|_{\ra M})=0\} \to L^2(M,\E^-)\ .$$
Let $\D_M(A)^-$ be the adjoint of $\D_M(A)^+$. Then $\D_M(A)=\left(\begin{array}{cc} 0 & \D_M(A)^- \\ \D_M(A)^+ & 0 \end{array}\right)$ is a selfadjoint operator on $L^2(M,\E)=L^2(M,\E^+)\oplus L^2(M,\E^-)$.

If $M$ is odd-dimensional, an operator $A$ as above is called a trivializing operator if in addition it is odd with respect to $\gradu_{\ra M}$. Then the operator $\D_M(A)$ is defined as the closure of
$$\dirac_M:\{f \in \C(M,\E)~|~ 1_{\ge 0}(\D_{\ra M}+A)(f|_{\ra M})=0\} \to L^2(M,\E)\ .$$

The operator $\D_M(A)$ is a regular selfadjoint Fredholm operator with compact resolvents. (This can be shown as in \cite{wu}). 
Let $i$ be the parity of the dimension of $M$.
From the Baaj-Julg picture of $KK$-theory via unbounded Kasparov modules \cite[\S 17.11]{bl} it follows that there is an induced class $[\D_M(A)] \in KK_i(\bbbc,\A)  \cong K_i(\A)$, called the index (class) of $\D_M(A)$. 

We also need cylindric index classes:

Let $\chi:M_{cyl} \to [0,1]$ be a smooth function with support in $Z_r$ such that $\chi|_{\{x_1 \ge \ - 3\ve/4\}}=1$. We define $\D^{cyl}_M(A)$ as the closure of 
$$\dira_{\E}- c(dx_1)\chi A:\C_c(M,\E) \to L^2(M,\E)$$ 
if $M$ is odd-dimensional and as the closure of 
$$\dira_{\E}- c(dx_1)\chi \gradu A:\C_c(M,\E) \to L^2(M,\E)$$ 
if $M$ is even-dimensional.  
Again, $\D_{M}^{cyl}(A)$ is a regular selfadjoint Fredholm operator (see for example \cite{waaps} for a detailed discussion) and thus defines an element in $KK_i(\bbbc,\A)$. Here the resolvents are non-compact, hence the Baaj-Julg picture does not apply. See \cite[Def. 2.4]{wa1} for the relevant definition of the Kasparov class that will be used in the following. 

The following equality has been essentially established  in the even case in \cite[\S 10]{llp} and follows in the odd case from \cite[\S 3.3]{lp6} together with \cite[Lemma 9.2]{wa1}. We give a different proof here, whose method will also be used in the proof of the product formula for index classes, Theorem \ref{prod}. It is similar to the proof of \cite[Theorem 7.2]{llp}.

\begin{prop}
\label{eqAPScyl}
In $KK_i(\bbbc,\A)$
$$[\D_M(A)]=[\D_M^{cyl}(A)] \ .$$
\end{prop}

\begin{proof}
We consider the case $i=1$. The even case is analogous with the obvious changes.

Recall that $p:Z \to \ra M$ is the projection. Endow $\E_Z=p^*\E_{\ra M}$ with the product Dirac bundle structure. Let $\dira_Z$ be the associated Dirac operator and denote by $\D_Z(A)$ the closure of $$\dira_Z- c(dx_1)A:\C_c(Z,\E_Z) \to L^2(Z,\E_Z) \ .$$ 

Furthermore let $Z_l=(-\infty,0] \times \ra M \subset Z$ and denote by $\D_{Z_l}(A)$ the closure of
$$\dira_Z-c(dx_1)A:\{f \in \C_c(Z_l,\E_Z)~|~ 1_{\ge 0}(\D_{\ra M}+A)(f|_{x_1=0})=0\} \to L^2(Z_l,\E_Z) \ .$$

The manifolds $Z_l$ and $M_{cyl}$ are obtained from $Z$ and $M$ by cutting and pasting along the hypersurfaces $x_1=-\ve/2$. By the relative index theorem (which is proven in \cite{bu} for manifolds without boundary and unperturbed Dirac operators, however the proof works here as well), 
$$[\D_M(A)- \chi c(dx_1) A]+[\D_Z(A)]=[\D_M^{cyl}(A)]+[\D_{Z_l}(A)] \ .$$

The operator $\D_{Z_l}(A)$ is invertible: Set $P=1_{\ge 0}(\D_{\ra M}+A)$ and $\sigma:=c(dx_1)$. Let $f \in \C_c(Z_l,\E_Z)$. We consider $f(x_1):=f|_{\{x_1\} \times \ra M}$ as an element in $\C(\ra M,\E_{\ra M})$. Then
\begin{align*}
(\D_{Z_l}(A)^{-1}f)(x_1) =&- \int_0^{x_1}e^{-(x_1-y_1)\D_{\ra M}(A)}(1-P) \sigma f(y_1)~dy_1 \\
&+  \int_{x_1}^{-\infty}e^{-(x_1-y_1)\D_{\ra M}(A)}P \sigma f(y_1)~dy_1 \ .
\end{align*}

The operator $\D_Z(A)$ is invertible as well. Hence $[\D_Z(A)]=[\D_{Z_l}(A)]=0$. The assertion follows since $[\D_M(A)- \chi c(dx_1)A]=[\D_M(A)]$. 
\end{proof}

Next we discuss Cartesian products:

Let $N$ be an oriented closed Riemannian manifold. Let $\E_N$ be a Dirac $\B$-bundle on $N$ and let $\dira_N:\C(N,\E_N) \to L^2(N,\E_N)$ be the associated Dirac operator. Its closure $\D_N$ induces an index class $[\D_N]\in KK_j(\bbbc,\B)$, where $j$ is the parity of the dimension of $N$.  

In the following we assume that $M$ and $N$ are even-dimensional. The other cases will be discussed below.

Let $\gradu_N$ be the grading operator on $\E_N$.

The bundle $\E_M \boxtimes \E_N$ is an $\bbbz/2$-graded hermitian $\A \ten \B$-bundle on $M \times N$ with grading operator $\gradu_{M\times N}=\gradu_M \gradu_N=\gradu_M\ten \gradu_N$ and with connection.

The product Dirac operator acting on $\C(M \times N,\E_M \boxtimes \E_N)$ is defined by $\dira_{M \times N}= \dira_M + \dira_N$. In order to illustrate our convention on the notation for tensor products we note that this equals $\dira_M \ten 1 + \gradu_M \ten \dira_N$.

We sketch how one sees that $\dira_{M \times N}$ is indeed a Dirac operator: For $f\in \C(M\times N)$ set $c_{M \times N}(df):=[\dira_{M\times N},f]$. Then for $v \in TM \subset T(M\times N)$ one has $c_{M\times N}(v)=c_M(v)$, and similarly for $v \in TN$. Using this one checks easily that $c_{M\times N}$ is a Clifford multiplication, endowed with which $\E_M \boxtimes \E_N$ becomes a Dirac $\A \ten \B$-bundle, and that $\dira_{M \times N}$ is the associated Dirac operator. 

In particular $c_{M \times N}(dx_1)=c_M(dx_1)$. 

Using the isomorphism $ic(dx_1):\E_M^+|_{\ra M} \cong \E_M^-|_{\ra M}$ we get an isomorphism
$$\Psi: \E_{\ra M} \boxtimes \E_N \stackrel{\cong}{\lr} ((\E_M^+ \boxtimes \E_N^+) \oplus (\E_M^- \boxtimes \E_N^-))|_{\ra M}=\E_{\ra(M \times N)}   \ .$$
It holds that
\begin{align}
\label{boundprod}
\dira_{\ra(M \times N)}&=\Psi(\gradu_N\dira_{\ra M} + \dira_N)\Psi^{-1}  \ . 
\end{align}

The operator $\hat A:=\Psi (\gradu_N A)\Psi^{-1}=\Psi(A \ten \gradu_N)\Psi^{-1}$ is a trivializing operator for $\dira_{\ra(M \times N)}$. Hence we get as above a Fredholm operator $\D_{M \times N}(\hat A)$, whose index is an element of $KK_0(\bbbc,\A \ten \B)$.

Our main result in this section expresses this index in terms of the indices of $\D_M(A)$ and $\D_N$ via the Kasparov product $$KK_*(\bbbc,\A) \times KK_*(\bbbc,\B) \to KK_*(\bbbc,\A \ten \B),~(a,b) \mapsto a \ten b \ .$$

We briefly recall its definition:
Let $D_1$ resp. $D_2$ be an odd selfadjoint operator with compact resolvents on a countably generated $\bbbz/2$-graded Hilbert $\A$ resp. $\B$-module $H_1$ resp. $H_2$.
Recall \cite[\S 18.9]{bl} that in the Baaj-Julg picture of $KK$-theory the Kasparov product $[D_1]\ten [D_2]$ is represented by the closure of the operator $D_1 +D_2$ whose domain (before taking closure) is the algebraic tensor product $\dom D_1 \ten \dom D_2$. Actually, this formula was the motivation for our definition of the product Dirac operator.

\begin{theorem}
\label{prod}
It holds that
$$[\D_M(A)] \ten [\D_N]=[\D_{M \times N}(\hat A)] \ .$$
\end{theorem}

\begin{proof}
By the comparing the above description of the Kasparov product with the definition of the product Dirac operator one sees that the class on the left hand side is represented by the closure $\Dk_{M \times N}(\hat A)$ of the odd operator 
$\dira_{M \times N}$ with domain $\dom \D_M(A) \ten \dom \D_N$ (understood as an algebraic tensor product). We use the method of the proof of Prop. \ref{eqAPScyl} in order show that $$[\Dk_{M \times N}(\hat A)]=[\D^{cyl}_{M \times N}(\hat A)] \ .$$ Then the assertion follows from Prop. \ref{eqAPScyl}.

Let $\Dk_{Z_l \times N}(\hat A)$ be the closure of the operator 
$\dira_{Z \times N} -  c(dx_1)\gradu_{Z\times N} \hat A$ with domain $\dom \D_{Z_l}(A) \ten \dom \D_N$, where $\dom \D_{Z_l}(A)$ is defined as in the proof of Prop. \ref{eqAPScyl}. 

The operator $\Dk_{Z_l \times N}(\hat A)$ is invertible with inverse 
$$\Dk_{Z_l \times N}(\hat A)^{-1}=\int_0^{\infty} \Dk_{Z_l \times N}(\hat A)e^{-t \D_{Z_l}(A)^2}e^{-\D_N^2} ~dt \ .$$
The integral converges for $t \to \infty$ since $\D_{Z_l}(A)$ is invertible, see the proof of Prop. \ref{eqAPScyl}.

Define $\D_{Z\times N}(\hat A)$ as the closure of 
$$\dira_{Z \times N} -  c(dx_1) \gradu_{Z\times N}\hat A: \C_c(Z \times N, \E_Z \boxtimes \E_N) \to L^2(Z \times N, \E_Z \boxtimes \E_N) \ .$$
By the relative index theorem
$$[\Dk_{M\times N}(\hat A)- \chi c(dx_1)\gradu_{M\times N} \hat A]+[\D_{Z\times N}(\hat A)]=[\D_{M\times N}^{cyl}(\hat A)]+[\Dk_{Z_l\times N}(\hat A)] \ .$$
Since $\D_{Z\times N}(\hat A)$ is also invertible, the assertion follows.
\end{proof}

\section{Products of unbounded Kasparov modules -- the remaining cases}
\label{produnbound}

Before discussing the cases in which $M$ and $N$ are not both even-dimensional we derive the general form of the Kasparov product for the remaining parities from its description in the even case given above. (It is needed here that the description remains valid if we deal with graded $C^*$-algebras.) The expressions we get for the product are the motivation for the definitions of the product Dirac operators in the following section.

Let $C_1$ be the Clifford algebra with one odd generator $\sigma$ fulfilling $\sigma^2=1$.  

The product involving odd $KK$-theory is defined via 
the isomorphism $KK_1(\bbbc, \A)\cong KK_0(\bbbc,\A \ten C_1)$. It maps a class $[D]$ represented by selfadjoint Fredholm operator $D$ on an ungraded countably generated Hilbert $\A$-module $H$ to the class $[\sigma D] \in KK_0(\bbbc,\A \ten C_1)$, where $\sigma D$ is defined on the $\bbbz/2$-graded Hilbert $\A \ten C_1$-module $H\ten C_1$. On the other hand given an odd selfadjoint Fredholm operator $D'$ on a $\bbbz/2$-graded Hilbert $\A \ten C_1$-module $H'$ and an odd involution $T$ on $H'$ with $TD'=D'T$, then the restriction of $TD'$ to the positive eigenspace of $T$ represents the preimage of $[D']$ under the above isomorphism. Note that right multiplication by the projection $\frac{1}{2}(1-\sigma)$ is trivial on the positive eigenspace of $T$, thus it is endowed with a canonical Hilbert $\A$-module structure. If $D'=\sigma D$ and $H'=H\ten C_1$ as before, we may choose $T=\sigma$ to get exactly the Kasparov module back we started with.

Let $D_1$ resp. $D_2$ be a selfadjoint operator with compact resolvents on a countably generated Hilbert $\A$- resp. $\B$-module $H_1$ resp. $H_2$.

\subsection{Even times odd}
First assume that $H_1$ is $\bbbz/2$-graded, $H_2$ is trivially graded, and $D_1$ is odd. We write $\gradu_1$ for the grading operator on $H_1$. The Kasparov product of  $[D_1]\in KK_0(\bbbc,\A)$ with $[\sigma D_2] \in KK_0(\bbbc,\B \ten C_1)$ is
$[D_1 + \sigma D_2] \in KK_0(\bbbc, \A\ten \B \ten C_1)$. 
We set $T=\sigma \gradu_1$. We have that $D_1 + \sigma D_2=\sigma \gradu_1(\sigma \gradu_1 D_1 +  \gradu_1 D_2)$ and that the positive eigenspace of $T$ equals $H_1 \ten H_2 \ten \bbbc (1 + \sigma)$. The choice of the base vector $\frac{1}{ 2}(1+\sigma)$ of $\bbbc(1+\sigma)$ defines an obvious isomorphism to $H_1 \ten H_2$. Here we consider $H_1 \ten H_2$ ungraded. The isomorphism intertwines $\sigma \gradu_1 D_1 + \gradu_1 D_2$ with  $D_1 +  \gradu_1 D_2$.
Thus
$$[D_1] \ten [D_2]= [D_1 + \gradu_1 D_2]\in KK_1(\bbbc,\A \ten \B) \ .$$

\subsection{Odd times even}
Now we assume that $H_2$ is $\bbbz/2$-graded, $H_1$ is trivially graded, and $D_2$ is odd. We write $\gradu_2$ for the grading operator on $H_2$. The Kasparov product of   $[\sigma D_1] \in KK_0(\bbbc,\A \ten C_1)$ with $[D_2]\in KK_0(\bbbc,\B)$ is
$[\sigma D_1 +  D_2] \in KK_0(\bbbc, \A\ten \B \ten C_1)$. 
Then $\sigma D_1 +  D_2=\sigma \gradu_2 (\gradu_2 D_1 + \sigma \gradu_2 D_2)$, and the positive eigenspace of $\sigma\gradu_2$ is $\bigl(H_1 \ten \bbbc (1+\sigma) \ten H_2^+\bigr)  \oplus  \bigl(H_1 \ten \bbbc (1-\sigma) \ten H_2^-\bigr)  \cong H_1 \ten H_2$. The last isomorphism intertwines $\gradu_2 D_1 + \sigma \gradu_2 D_2$ with $\gradu_2 D_1 +   D_2$.
Thus
$$[D_1] \ten [D_2]= [\gradu_2 D_1 + D_2]\in KK_1(\bbbc,\A \ten \B) \ .$$

\subsection{Odd times odd}

Now let $H_1, H_2$ be trivially graded. We write $C_1', C_1''$ for two copies of $C_1$ with generators $\sigma', \sigma''$ respectively. 

The class $$[\sigma' D_1]\ten [\sigma'' D_2] \in KK_0(\bbbc, \A \ten \B \ten C_1' \ten C_1'')$$ 
is represented by the odd operator $\sigma'D_1 + \sigma''D_2$ on $H_1 \ten H_2 \ten C_1' \ten C_1''$. 

Note that $\frac 12 (1 + i\sigma'\sigma'')$ is a rank one projection. By Morita equivalence the homomorphism $$p:\bbbc \to C_1' \ten C_1'', ~x \mapsto \frac 12 x(1 + i\sigma'\sigma'')$$ induces an isomorphism $p_*:KK_0(\bbbc, \A \ten \B ) \to KK_0(\bbbc, \A \ten \B\ten C_1' \ten C_1'')$. 

We define a representative of the preimage of $[\sigma'D_1 + \sigma''D_2]$ under $p_*$. The algebra $C_1' \ten C_1''$ acts on $\bbbc^2$ via the isomorphism $$C_1' \ten C_1'' \to M_2(\bbbc) \ ,$$
$$\sigma' \mapsto \Gamma_1:=\left(\begin{array}{cc} 1 &  0\\ 0 & -1\end{array}\right),~ \sigma''\mapsto \Gamma_2:=\left(\begin{array}{cc} 0 & i\\ -i & 0 \end{array}\right) \ .$$
The action is compatible with the grading if on  $\bbbc^2$ the grading defined by the operator $$-i\Gamma_1\Gamma_2=\left(\begin{array}{cc} 0 & 1\\ 1 & 0 \end{array}\right) \ .$$ 

In the following we show that the odd operator $\Gamma_1 D_1 + \Gamma_2 D_2$ on $H_1 \ten H_2 \ten \bbbc^2$ represents the preimage.

Define the Hilbert $C_1' \ten C_1''$-module $V:=\frac 12 (1 + i\sigma'\sigma'')(C_1' \ten C_1'')$.

The unit vector $e_1:=\frac 12 (1 + i\sigma'\sigma'')$ spans $V^+$, and the unit vector $e_2:=\frac{1}{2}(\sigma' -i\sigma'')$ spans $V^-$.

Note that canonically $\bbbc^2 \ten_p (C_1' \ten C_2'') \cong \bbbc^2 \ten V$.

Choose a unit vector $v_1\in (\bbbc^2)^+$ and let $v_2:=\Gamma_1v_1 \in (\bbbc^2)^-$. The even isomorphism of Hilbert $C_1' \ten C_2''$-modules
$$\bbbc^2 \ten V \to C_1' \ten C_2'' \ ,$$ 
$$v_1 \ten e_1 \mapsto e_1, \quad v_1 \ten e_2 \mapsto e_2 \ ,$$
$$v_2 \ten e_1 \mapsto \sigma'e_1, \quad v_2 \ten e_2 \mapsto \sigma'e_2 \ ,$$ is compatible with the left $C_1' \ten C_2''$-action on both spaces. Summarizing, we get an isomorphism $H_1 \ten H_2 \ten \bbbc^2\ten_p (C_1' \ten C_2'')\cong H_1 \ten H_2 \ten C_1' \ten C_1''$ intertwining $\Gamma_1 D_1 + \Gamma_2D_2$ and $\sigma'D_1 + \sigma''D_2$.

Thus
$$[D_1] \ten [D_2]=[\Gamma_1 D_1 + \Gamma_2D_2] \in KK_0(\bbbc, \A \ten \B)\ .$$

(This calculation corrects a similar but flawed argument in the proof of \cite[Lemma 9.2]{wa1})

\section{Product structures for Dirac operators -- the remaining cases}

\subsection{$M$ is even-dimensional and $N$ odd-dimensional}
\label{evodd}

Let $\gradu_M$ be the grading operator on $\E_M$. The bundle $\E_M \boxtimes \E_N$ is now considered an ungraded $\A \ten \B$-vector bundle. The product Dirac operator is defined as $$\dira_{M \times N}= \dira_M + \gradu_M\dira_N \ .$$
Hence here also $c_{M\times N}(dx_1)=c_M(dx_1)$.

The isomorphism $ic(dx_1):\E_M^+|_{\ra M} \cong \E_M^-|_{\ra M}$ induces an isomorphism
$$\Psi: (\E_{\ra M}\oplus \E_{\ra M}) \boxtimes \E_N \stackrel{\cong}{\lr} (\E_M \boxtimes \E_N)|_{\ra M}=\E_{\ra(M \times N)}   \ .$$
We let the matrices $\Gamma_1$, $\Gamma_2$, which were defined in \S \ref{produnbound}, act on $(\E_{\ra M}\oplus \E_{\ra M})\boxtimes \E_N$. 

Then 
\begin{align}
\label{boundprodevodd}
\dira_{\ra(M \times N)}&=\Psi(\Gamma_1\dira_{\ra M} + \Gamma_2\dira_N)\Psi^{-1}  \ . 
\end{align}

Theorem \ref{prod} holds in this situation for $\hat A:=\Psi\Gamma_1 (A \ten 1)\Psi^{-1}$.

\subsection{$M$ is odd-dimensional and $N$ even-dimensional}
\label{oddev}

In analogy to the previous case the bundle $\E_M \boxtimes \E_N$ is considered ungraded and the product Dirac operator is defined as $$\dira_{M \times N}:=\gradu_N\dira_M + \dira_N \ .$$
It follows that $c_{M\times N}(dx_1)=\gradu_Nc_M(dx_1)$.

We have that
$$\E_{\ra(M \times N)} = (\E_M \boxtimes \E_N)|_{\ra M} = \E_{\ra M} \boxtimes \E_N  \ ,$$
which is a graded vector bundle with grading operator $\gradu_{\ra M \times N}=ic_M(dx_1)\gradu_N=\gradu_{\ra M} \gradu_N$.

Then 
\begin{align}
\label{boundprododdev}
\dira_{\ra(M \times N)}&=\dira_{\ra M} -i \gradu_N\dira_N  \ . 
\end{align}

Theorem \ref{prod} holds with $\hat A:=A \ten 1$.

\subsection{$M,~N$ are odd-dimensional.}
\label{oddodd}

Consider the bundle $\E_{M\times N}:=(\E_M\oplus \E_M) \boxtimes \E_N$, on which $\Gamma_1, \Gamma_2$ from \S \ref{produnbound} act. The associated product Dirac operator is defined by $$\dira_{M \times N}= \Gamma_1\dira_M + \Gamma_2\dira_N$$ and the grading is given by $\gradu_{M \times N}=-i\Gamma_1\Gamma_2$. We see that $c_{M\times N}(dx_1)=\Gamma_1c_M(dx_1)$. We have an isomorphism
\begin{align*}
\Psi:\E_{\ra M}\boxtimes \E_N &\to \E_{\ra(M\times N)}=\E_{M\times N}^+|_{\ra M}\\
x \ten y &\mapsto \frac{1}{\sqrt 2}(x,x)\ten y \ .
\end{align*}

Then
\begin{align}
\label{boundprododdodd}
\dira_{\ra(M \times N)}&=\Psi(\dira_{\ra M} + \gradu_{\ra M}\dira_N)\Psi^{-1}  \ . 
\end{align}

Theorem \ref{prod} holds with $\hat A:=\Psi(A \ten 1)\Psi^{-1}$.

\section{Product formula for twisted signature classes}
\label{prodsig}

Let $\F_M$ be a flat hermitian $\A$-vector bundle on $M$ endowed with a compatible flat connection and let $\F_{\ra M}=\F_M|_{\ra M}$. We assume that $\F_M|_{U_{\ve}}=p^*\F_{\ra M}$ as a hermitian vector bundle and that the connection is of product type on $U_{\ve}$. Analogously let $\F_N$ be a flat $\B$-vector bundle on $N$, also endowed with a hermitian structure and a compatible flat connection. 

We denote by $\Omega^*(M,\F_M)$ the space of smooth twisted de Rham forms with de Rham differential $d_M$. Let $\Omega_{(2)}^*(M,\F_M)$ the Hilbert $\A$-module of $L^2$-forms.

We endow $\Lambda^*T^*M$ with the Levi-Civit\`a connection. Thus we have an induced connection on $\Lambda^*T^*M \ten \F$.
The bundle $\Lambda^*T^*M\ten \F_M$ is a Clifford module with Clifford multiplication $c_M(\alpha)\omega =\alpha\wedge \omega-\iota(\alpha)\omega$. Recall that the induced chirality operator $\tau_M$ is a selfadjoint involution on $\Lambda^*T^*M\ten \F$, see \cite[Lemma 3.17]{bgv}. We denote by $\Lambda^{\pm} T^*M \ten \F_M$ resp. $\Omega^{\pm}(M,\F_M)$ the eigenspace associated to the eigenvalue $\pm 1$ of $\tau_M$. If $M$ is even-dimensional we endow $\Lambda^*T^*M \ten \F$ with the $\bbbz/2$-grading induced by $\tau_M$. With these structures $\Lambda^*T^*M\ten \F_M$ is a Dirac bundle. The signature operator is defined as the associated Dirac operator, see \cite[\S 3.6]{bgv}.

We fix the isometry $$\Phi_M:\Lambda^* T^*\ra M \ten \F_{\ra M} \to (\Lambda^+ T^*M|_{\ra M}) \ten \F_{\ra M},~\alpha \mapsto \frac{1}{\sqrt 2}\bigl(dx_1 \wedge \alpha + \tau_M(dx_1 \wedge \alpha)\bigr) \ .$$

\subsection{The even case}
\label{even}

In the following we assume that $M$ is even-dimensional. 

For $\alpha \in \Lambda^* T^*\ra M$ 
$$\tau_M(dx_1 \wedge \alpha)= \tau_{\ra M} \alpha$$
and $$\tau_M(\alpha)=dx_1 \wedge \tau_{\ra M}(\alpha) \ .$$

The signature operator on $\Omega^*(M,\F_M)$ equals
$$d^{sign}_M:=d_M+d_M^*=d_M - \tau_M d_M \tau_M\ .$$

Note that the normalization here is as in \cite[\S 3.6]{bgv} and differs from \cite{hs}\cite{llp}. The corresponding index classes agree up to sign, see \S \ref{normalization}. Accordingly, also our convention in the odd case is different.

It holds that 
\begin{align}
\label{boundsign}
B(d_M^{sign})&=\Phi_M(d_{\ra M}\tau_{\ra M} + \tau_{\ra M}d_{\ra M})\Phi_M^{-1} \ .
\end{align}

We denote the closure of $d_{\ra M}\tau_{\ra M} + \tau_{\ra M}d_{\ra M}:\Omega^*(\ra M,\F_{\ra M}) \to \Omega^*_{(2)}(\ra M,\F_{\ra M})$ by $\D^{bd}_{\ra M}$. In order to avoid confusion, we point out that $\D^{bd}_{\ra M}$ agrees with the odd signature operator in the convention of some authors, but not in convention used here. For the precise relation see \S \ref{oddcase}.

The following definition generalizes the boundary conditions considered in \cite[\S 6.3]{lp6}.

\begin{ddd}
\label{definvol}
Assume that there is an orthogonal decomposition $\Omega^*_{(2)}(\ra M,\F_{\ra M})=V\oplus W$ with respect to which $\tau_{\ra M}$ and $\D^{bd}_{\ra M}$ are diagonal. Furthermore assume that $\D^{bd}_{\ra M}|_V$ is invertible. Let $\inv$ be an operator on $\Omega^*_{(2)}(\ra M,\F_{\ra M})$ that vanishes on $V$, is an involution on $W$ and anticommutes with $\tau_{\ra M}$ and $\D^{bd}_{\ra M}$. 

We call a trivializing operator $A$ of $B(d_M^{sign})$ {\rm symmetric} with respect to $\inv$ if it is diagonal with respect to the decomposition $\Phi_M(V)\oplus \Phi_M(W)$, vanishes on $\Phi_M(V)$ and anticommutes with $\Phi_M\inv\Phi_M^{-1}$.

If $A$ is a symmetric trivializing operator, then the index class $$\sigma_{\inv}(M,\F_M):=[\D_M^{sign}(A)] \in K_0(\A)$$ is called the {\rm (twisted) signature class}.

We call  the symmetric trivializing operator $A_{\inv}:=i\Phi_M \inv\tau_{\ra M}\Phi_M^{-1}$ the {\rm canonical} symmetric trivializing operator of $B(d_M^{sign})$ with respect to $\inv$.  
\end{ddd}

Since $(\D^{bd}_{\ra M}+i\inv \tau_{\ra M})^2=(\D^{bd}_{\ra M})^2+\inv^2$ is invertible, the operator $\D^{bd}_{\ra M}+i\inv\tau_{\ra M}$ is invertible as well. Hence $A_{\inv}$ is indeed a trivializing operator for $B(d_M^{sign})$.  

In the following we extend any operator on $W$ tacitly to $\Omega^*_{(2)}(\ra M,\F_{\ra M})$ by letting it vanish on $V$.

The following result sharpens and generalizes similar calculations in \cite{lp4}. 

\begin{lem}
\label{independ}
The twisted signature class $\sigma_{\inv}(M,\F_M)$ does not depend on the choice of the symmetric trivializing operator.  
\end{lem}

\begin{proof}

Let $A_0, A_1$ be two trivializing operators for $B(d_M^{sign})$ that are symmetric with respect to $\inv$. 

Consider the cylinder $Z:=\bbbr \times \ra M$ and let $\D_Z^{sign}$ be the signature operator on $\Omega_{(2)}^*(Z,p^*\F_{\ra M})$. Recall that the positive and negative eigenspace of $\tau_Z$ are identified via $ic(dx_1)$. We get translation invariant spaces $$\tilde V=L^2(\bbbr)\ten \Phi_Z(V)\ten\bbbc^2$$
$$\tilde W:=L^2(\bbbr)\ten \Phi_Z(W)\ten\bbbc^2$$ such that $\Omega^*_{(2)}(Z,p^*\F_{\ra M})=\tilde V\oplus \tilde W$. The operators $A_i$ and $\Phi_Z\inv\Phi_Z^{-1}$ define translation invariant operators on $\Omega^*_{(2)}(Z,p^*\F_{\ra M})$. Note that $\D_Z^{sign}$ is invertible on $\tilde V$ since $\D^{bd}_{\ra M}|_{\ra M}$ is invertible on $V$.

Let $\chi_0,\chi_1:Z \to [0,1]$ be smooth functions such that  $\chi_0(x_1,x_2)=1$ if $x_1 \le 0$ and $\chi_0(x_1,x_2)=0$ if $x_1 \ge \frac 12$ and that $\chi_1(x_1,x_2)=1$ if $x_1 \ge 1$ and $\chi_1(x_1,x_2)=0$ if $x_1 \le \frac 12$.

Prop. \ref{eqAPScyl} and the relative index theorem \cite{bu} imply that
$$[\D_M^{sign}(A_0)] + [(\D_Z^{sign} - c(dx_1)\tau_Z(\chi_0 A_0  + \chi_1 A_1))|_{\tilde W}]=[\D_M^{sign}(A_1)] \ .$$

Let $j:\bbbc \to C_1$ be the unique unital homomorphism.
It holds that $[j] \in KK_0(\bbbc,C_1)=0$, thus $$\Ima(j^*:KK_0(C_1,\A) \to KK_0(\bbbc,\A))=0 \ .$$

There is an even unital homomorphism $C_1 \to B(\tilde W)$ mapping $\sigma$ to $ic(dx_1)\tau_Z(\Phi_Z \inv \Phi_Z^{-1})$. 
Since $(\D_Z^{sign} - c(dx_1)\tau_Z(\chi_0 A_0  + \chi_1 A_1))|_{\tilde W}$ anticommutes with $ic(dx_1)\tau_Z(\Phi_Z \inv \Phi_Z^{-1})$, we have that
$$[(\D_Z^{sign} - c(dx_1)\tau_Z(\chi_0 A_0  + \chi_1 A_1))|_{\tilde W}]\in  \Ima(j^*)\ .$$
\end{proof}

Note that for $\inv^{opp}:=-i\inv\tau_{\ra M}$ the signature class $\sigma_{\inv^{opp}}(M,\F_M)$ is well-defined and that $A_{\inv^{opp}}=\Phi_M \inv\Phi_M^{-1}$. Since $\inv^{opp} \inv=i\tau_{\ra M}$, the second assertion of the following Lemma implies that 
\begin{align}
\label{eqinvinvopp}
\sigma_{\inv}(M,\F_M)&=\sigma_{\inv^{opp}}(M,\F_M) \ .
\end{align}

\begin{lem} For $j=0,1$ let $\Omega_{(2)}^*(\ra M,\F_{\ra M})=V_j\oplus W_j$ be an orthogonal decomposition and let $\inv_j$ be an involution on $W_j$ such that $\sigma_{\inv_j}(M,\F_M)$ is well-defined.
\begin{enumerate}
\item Assume that $W_1 \subset W_0$ and $\inv_1=\inv_0|_{W_1}$.
\item Assume that $W:=W_0=W_1$. Let $E^+$ be the positive and $E^-$ the negative eigenspace of $\tau_{\ra M}$ on $W$. We identify $E^-$ with $E^+$ using the isomorphism $\inv_0:E^- \to E^+$. Then there is a unitary $u$ on $E^+$ such that with respect to the decomposition $W=E^+ \oplus E^-$ 
$$\inv_1=\left(\begin{array}{cc} 0 & u^* \\ u & 0 \end{array}\right) \ .$$
We assume that the spectrum of $u$ is not equal to $S^1$.
\end{enumerate}

If one of the previous two conditions holds, then
$$\sigma_{\inv_0}(M,\F_M)=\sigma_{\inv_1}(M,\F_M) \ .$$
\end{lem}

\begin{proof}
In the first case we get the equality since any trivializing operator that it symmetric with respect to $\inv_1$ is also symmetric with respect to $\inv_0$. 

Now assume (2). Since the spectrum of $u$ is not equal to $S^1$, there is a selfadjoint operator $a$ on $E^+$ such that $u=e^{ia}$. Set $u_t=e^{ita},~t \in [0,1]$. The involutions $\inv_0,\inv_1$ are homotopic to each other via the path of involutions
$$\inv_t=\left(\begin{array}{cc} 0 & u_t^* \\ u_t & 0 \end{array}\right) \ .$$
Since $\D^{bd}_{\ra M}$ anticommutes with $\inv_0$ and commutes with $\tau_{\ra M}$, we get that $$\D^{bd}_{\ra M}=\left(\begin{array}{cc} D & 0 \\ 0 & -D \end{array}\right)$$ with $D=(\D^{bd}_{\ra M})|_{E^+}$. Furthermore $\D^{bd}_{\ra M}$ also anticommutes with $\inv_1$. This implies that $D$ commutes with $u$ and $u^*$. Hence it commutes also with $u_t$ and $u_t^*$. It follows that $\D^{bd}_{\ra M}$ anticommutes with $\inv_t$.
Thus $\sigma_{\inv_t}(M,\F_M)$ is well-defined. By the homotopy invariance of $KK$-theory classes it does not depend on $t$.
\end{proof}

The following proposition generalizes both cases of the previous Lemma:

\begin{prop}
\label{propev}
For $j=0,1$ let $\Omega^*_{(2)}(\ra M,\F_{\ra M})=V_j\oplus W_j$ be an orthogonal decomposition and let $\inv_j$ be an involution on $W_j$ such that $\sigma_{\inv_j}(M,\F_M)$ is well-defined.
Assume that $V_0=(V_0 \cap V_1) \oplus (V_0 \cap W_1)$ and
$W_0=(W_0\cap V_1) \oplus (W_0 \cap W_1)$ and that $\inv_0,~ \inv_1$ restrict to involutions on $W_0 \cap W_1$. Let $\inv_0|_{W_0\cap W_1}$ and $\inv_1|_{W_0\cap W_1}$ fulfill condition (2) of the previous Lemma. Then
$$\sigma_{\inv_0}(M,\F_M)=\sigma_{\inv_1}(M,\F_M) \ .$$
\end{prop}

\begin{proof}
Set $\tilde \inv_j=\inv_j|_{W_0\cap W_1}$. Note that $\sigma_{\tilde \inv_j}(M,\F_M)$ is well-defined. By part (1) of the previous Lemma $\sigma_{\inv_j}(M,\F_M)=\sigma_{\tilde \inv_j}(M,\F_M)$ and part (2) implies that $\sigma_{\tilde \inv_0}(M,\F_M)=\sigma_{\tilde \inv_1}(M,\F_M)$.
\end{proof}

Now, following \cite{lp6}, we introduce the particular involution that is used for the definition of the signature class. For brevity it will be denoted $\alpha_M$ though it depends only on the structures on $\ra M$. 

Let $m=\dim M/2$.

Let $V_M$ be the closure of $d^*\Omega^m(\ra M,\F_{\ra M})\oplus d\Omega^{m-1}(\ra M,\F_{\ra M})$
and $W_M=V_M^{\perp}$. Define $\Omega^{<}_{\F_{M}}$ as the closed subspace of $W_M$ spanned by forms of degree smaller than or equal to $m-1$ and correspondingly define $\Omega^{>}_{\F_{M}}$ as the subspace spanned by forms of degree bigger than or equal to $m$.

We make the following assumption:

\begin{ass} 
\label{ass}
The closure of $d:\Omega^{m-1}(\ra M,\F_{\ra M}) \to \Omega_{(2)}^{m}(\ra M,\F_{\ra M})$ has closed range.
\end{ass}

Note that the operators $\tau_{\ra M},d_{\ra M},d_{\ra M}^*$ restrict to operators on $V_M$ resp. $W_M$ and that $\tau_{\ra M}:\Omega^{<}_{\F_M} \to \Omega^{>}_{\F_M}$ is an isomorphism. 

Assumption \ref{ass} implies that $\D^{bd}_{\ra M}$ is invertible on $V_M$ and that
$$V_M \oplus W_M=\Omega^*_{(2)}(\ra M,\F_{\ra M}) \ .$$ 

Let $\alpha_M$ be the involution on $W_M$ with positive eigenspace $\Omega^{<}_{\F_{M}}$ and negative eigenspace $\Omega^{>}_{\F_{M}}$. Then $\D^{bd}_{\ra M}$ and $\alpha_M$ 
anticommute.

We write $\sigma(M,\F_M):=\sigma_{\alpha_M}(M,\F_M)$.

Note that Assumption \ref{ass} does not depend on the choice of the Riemannian metric since $\Omega_{(2)}^{m}(\ra M,\F_{\ra M})$ as a topological vector space does not depend on the Riemannian metric. Using the homotopy invariance of $KK$-theory classes one also shows that $\sigma(M,\F_M)$ does not depend on the choice of the Riemannian metric.

The following technical lemma will be needed when we apply Prop. \ref{propev}.  

\begin{lem}
\label{dirsum}
Assume that $N$ is even-dimensional. Let the de Rham operators on $\Omega^*(\ra M,\F_{\ra M})$ and on $\Omega^*(\ra M \times N,\F_{\ra M}\boxtimes \F_N)$ fulfill Assumption \ref{ass}. We have that
\begin{align*}
V_{M\times N}&=(V_M \ten \Omega_{(2)}^*(N,\F_N))\cap V_{M\times N} \oplus  (W_M \ten \Omega_{(2)}^*(N,\F_N))\cap V_{M\times N}\\
W_{M\times N}&=(V_M \ten \Omega_{(2)}^*(N,\F_N))\cap W_{M\times N} \oplus  (W_M \ten \Omega_{(2)}^*(N,\F_N))\cap W_{M\times N} \ .
\end{align*}

The operator $\D^{bd}_{\ra M\times N}$ is diagonal with respect to the decompositions on the right hand side and is invertible on $V_M \ten \Omega_{(2)}^*(N,\F_N)$.
\end{lem}

\begin{proof}
Note first that $d_{\ra M\times N}, d_{\ra M\times N}^*$ and $\tau_{\ra M\times N}$ map the spaces $V_M \ten \Omega_{(2)}^*(N,\F_N)$ and $W_M \ten \Omega_{(2)}^*(N,\F_N)$ to themselves.

For each $k$
\begin{align*}
\lefteqn{\Omega^k(\ra M \times N,\F_{\ra M} \boxtimes \F_N)}\\
&=\Omega^k(\ra M \times N,\F_{\ra M} \boxtimes \F_N)\cap (V_M \ten \Omega_{(2)}^*(N,\F_N))  \\
&\quad \oplus \Omega^k(\ra M \times N,\F_{\ra M} \boxtimes \F_N)\cap (W_M \ten \Omega_{(2)}^*(N,\F_N)) \ .
\end{align*}
Hence we only need to consider the degrees $k:=(\dim M+\dim N)/2$ and $k-1$.

We begin by proving the first equation:
Let $\gamma=d(\alpha \wedge \beta)\in d\Omega^{k-1}(\ra M\times N,\F_M \boxtimes \F_N)\subset V_{M\times N}$ with $\alpha \in \Omega^*(\ra M,\F_{\ra M}),~\beta \in \Omega^*(N,\F_N)$. If $\alpha \in W_M$, then $d\alpha \in W_M$, thus $\gamma \in (W_M \ten \Omega_{(2)}^*(N,\F_N))\cap V_{M\times N}$. If $\alpha \in V_M$, then $d\alpha \in V_M$, hence $\gamma \in (V_M \ten \Omega_{(2)}^*(N,\F_N))\cap V_{M\times N}$. An analogous consideration works for $d^*\Omega^k(\ra M\times N,\F_M \boxtimes \F_N)$. This shows the first equation.

For the proof of the second equation let $\gamma \in \Omega^{k-1}(\ra M\times N,\F_M \boxtimes \F_N) \cap W_{M\times N}$. Hence $d\gamma=0$. Write $\gamma=\gamma_1 +\gamma_2$ with $\gamma_1 \in W_M \ten \Omega_{(2)}^*(N,\F_N)$, $\gamma_2 \in V_M \ten \Omega_{(2)}^*(N,\F_N)$. Then $d\gamma_1 \in W_M \ten \Omega_{(2)}^*(N,\F_N)$, $d\gamma_2 \in V_M \ten \Omega_{(2)}^*(N,\F_N)$. Since these spaces are orthogonal to each other, the equation $d(\gamma_1+\gamma_2)=0$ implies that $d\gamma_1=d\gamma_2=0$. Thus $\gamma_1,\gamma_2 \in W_{M\times N}$. The case $\gamma \in \Omega^{k}(\ra M\times N,\F_M \boxtimes \F_N)$ with $d^*\gamma=0$ can be treated analogously. Now the second equation follows. 

The operator $\D_{\ra M \times N}^{bd}$ respects the decompositions on the right hand side since $d$ and $\tau_{\ra M\times N}$ do. Its square is the Laplace operator $\Delta_{\ra M \times N}=\Delta_{\ra M} + \Delta_N$. Since $\Delta_{\ra M}$ is invertible on $V_M$, the operator $\Delta_{\ra M\times N}$ is invertible on $V_M \ten \Omega_{(2)}^*(N,\F_N)$. Hence also $\D^{bd}_{\ra M\times N}$ is invertible on $V_M \ten \Omega_{(2)}^*(N,\F_N)$. 
\end{proof}

\begin{theorem}
\label{prodev}
Let $M, N$ be even-dimensional.

If Assumption \ref{ass} holds for the de Rham operators on $\Omega^*(\ra M,\F_{\ra M})$ and on $\Omega^*(\ra M \times N,\F_{\ra M}  \boxtimes \F_N)$, then 
$$\sigma(M,\F_M) \ten \sigma(N,\F_N)=\sigma(M \times N,\F_M \boxtimes \F_N) \in K_0(\A \ten \B) \ .$$
\end{theorem}

\begin{proof}
We denote by $\Gamma_M$ the grading operator with respect to the $\bbbz/2$-grading determined by the parity of the degree of a differential form on $M$. 

The de Rham operator on $M \times N$ fulfills
\begin{align*}
d_{M \times N}= d_M \ten 1 + \Gamma_M \ten d_N \ .
\end{align*}
Thus
\begin{align}
\label{signprod}
d_{M \times N}+d_{M\times N}^*&=(d_M+d_M^*)\ten 1 + \Gamma_M \ten (d_N+d_N^*) \ .
\end{align}

Note for later that these two equations also hold for $M$ or $N$ odd-dimensional.

We begin by proving the theorem for closed $M$. We conclude (recall our convention on graded tensor products) that
$$\D_{M \times N}^{sign}=\D_M^{sign}\ten 1 + \Gamma_M \ten \D_N^{sign}=\D_M^{sign} + \Gamma_M \tau_M \D_N^{sign} \ .$$
Furthermore
\begin{align*}
\tau_{M \times N}=\tau_M \tau_N \ .
\end{align*}

We fix the following notation: Let $D$ be an odd selfadjoint Fredholm operator on a $\bbbz/2$-graded countably generated Hilbert $\A$-module $H$ and let $I$ be a unitary on $H^-$. We define the symmetrized product $\sym(I,D)=\left(\begin{array}{cc} 0 & D^-I^* \\ ID^+ & 0\end{array}\right)$. Then $\sym(I,D)$ is a regular selfadjoint odd Fredholm operator and $[\sym(I,D)]=[D] \in KK_0(\bbbc,\A)$ by the additivity of the Fredholm index. If $I$ is an even unitary defined on $H$ that commutes with $D$, then $\sym(I|_{H^-},D)=ID$.

Applying this property twice with $I=\Gamma_M\tau_M$ yields that in $KK_0(\bbbc,\A \ten \B)$
\begin{align*}
[\D_M^{sign}] \ten [\D_N^{sign}]&=[\Gamma_M\tau_M\D_M^{sign}] \ten [\D_N^{sign}]\\
&=[\Gamma_M\tau_M\D_M^{sign} +  \D_N^{sign}]\\
&=[\Gamma_M\tau_M(\D_M^{sign} + \Gamma_M\tau_M \D_N^{sign})]\\
&=[\D_M^{sign} + \Gamma_M\tau_M \D_N^{sign}]\\
&=[\D_{M \times N}^{sign}] \ .
\end{align*}
The second equality follows from the description of the Kasparov product before Theorem \ref{prod}.

Now we consider the case where $M$ is a manifold with boundary.

Define the involution
$$\tilde \alpha_M:=(\Phi_{M\times N}^{-1}\circ\Psi\circ \Phi_M)\alpha_M\tau_N (\Phi_M^{-1}\circ \Psi^{-1}\circ\Phi_{M\times N})$$ 
on $$\tilde W_M:=(\Phi_{M\times N}^{-1}\circ\Psi\circ\Phi_M)(W_M \ten \Omega_{(2)}^*(N,\F_N))$$ 
and set 
$$\tilde V_M:=(\Phi_{M\times N}^{-1}\circ\Psi \circ \Phi_M)(V_M \ten \Omega_{(2)}^*(N,\F_N)) \ .$$ 

\begin{sublem}
\label{sublevev}
\begin{enumerate}
\item It holds that 
\begin{align*}
\tilde V_M&=V_M \ten \Omega_{(2)}^*(N,\F_N) \\
\tilde W_M&=W_M \ten \Omega_{(2)}^*(N,\F_N) 
\end{align*}
and that $\tilde \alpha_M=\alpha_M$.
\item The operator $\tilde \alpha_M$ anticommutes with $\D^{bd}_{\ra M\times N}$ and $\tau_{\ra M\times N}$ and commutes with $\alpha_{M\times N}$.
\end{enumerate}
\end{sublem}

\begin{proof}
For $\alpha \in \Lambda^*T^*\ra M\ten \F_M,~ \beta \in \Lambda^* T^*N \ten \F_N$
\begin{align*}
\lefteqn{(\Psi \circ \Phi_M)(\alpha \wedge \beta)}\\
&=\frac 1{\sqrt 2}\Psi(dx_1 \wedge \alpha \wedge \beta + \tau_M(dx_1 \wedge \alpha) \wedge \beta)\\
&=\frac{1}{2\sqrt 2}\bigl(dx_1 \wedge \alpha + \tau_M(dx_1 \wedge \alpha)) \wedge (\beta + \tau_N \beta) + i(-\alpha + \tau_M(\alpha))\wedge (\beta - \tau_N \beta)\bigr) \\
&=\frac{1}{2\sqrt 2}\bigl(dx_1 \wedge \alpha \wedge (\beta + \tau_N \beta) + idx_1\wedge \tau_{\ra M}(\alpha) \wedge (\beta-\tau_N\beta)\bigr) + \tau_{M \times N}(\dots) \ .
\end{align*}

Here the dots represent a repetition of the first summand, such that the last line is in the positive eigenspace of $\tau_{M\times N}$.

Thus
$$(\Phi_{M\times N}^{-1}\circ\Psi \circ \Phi_M)(\alpha \wedge \beta)=\frac{1}{2}\bigl(\alpha \wedge (\beta + \tau_N \beta) + i\tau_{\ra M}\alpha \wedge (\beta-\tau_N\beta)\bigr) \ .$$

In particular
$$(\Phi_{M\times N}^{-1}\circ\Psi \circ \Phi_M)(\alpha  \wedge (\beta +\tau_N \beta))=\alpha \wedge (\beta + \tau_N \beta)$$
and 
\begin{align*}
(\Phi_{M\times N}^{-1}\circ\Psi \circ \Phi_M)(\tau_{\ra M} \alpha  \wedge (\beta -\tau_N \beta))
&=i\alpha\wedge (\beta- \tau_N \beta) \ .
\end{align*}

Let $\tilde W_M^{\pm}$ be the positive resp. negative eigenspace of $\tilde \alpha_M$. It follows that
\begin{align*}
\tilde W_M^+&=\Omega^{<}_{\F_M} \ten \Omega_{(2)}^*(N,\F_N)\\ 
\tilde W_M^-&=\Omega^{>}_{\F_M} \ten \Omega_{(2)}^*(N,\F_N) \ .
\end{align*}
This shows the second  and third equality of assertion (1). The first equality follows
since $\tilde V_M$ is the orthogonal complement of $\tilde W_M$.

Furthermore $\tau_{\ra M\times N}$ interchanges the spaces $\Omega^{<}_{\F_M} \ten \Omega_{(2)}^*(N,\F_N)$ and $\Omega^{>}_{\F_M} \ten \Omega_{(2)}^*(N,\F_N)$ whereas $d_{\ra M \times N}$ preserves them. This implies assertion (2).
\end{proof}

We set $\inv=\tilde \alpha_M$. By the Sublemma $\sigma_{\inv}(M\times N,\F_M\boxtimes \F_N)$ is well-defined. One checks easily that $\tilde \alpha_M$ and $\alpha_{M\times N}$ restrict to involutions on $\tilde W_M \cap W_{M\times N}$. Since on that space $(\tilde \alpha_M\alpha_{M \times N})^2=1$, the spectrum of the restriction of $\tilde \alpha_M\alpha_{M \times N}$ to $\tilde W_M \cap W_{M\times N}$ is contained in $\{-1,1\}$. Hence, by Lemma \ref{dirsum} and Prop. \ref{propev},
$$\sigma_{\inv}(M\times N,\F_M\boxtimes \F_N)=\sigma(M\times N,\F_M\boxtimes \F_N) \ .$$

Let $A_M$ be the canonical symmetric trivializing operator for $B(d_M^{sign})$ with respect to $\alpha_M$. 

By definition $\hat A_M=i(\Psi \circ \Phi_M)(\alpha_M \tau_{\ra M} \tau_N)(\Psi \circ \Phi_M)^{-1}$. Hence
$$(\Phi_{M\times N})^{-1}\hat A_M \Phi_{M\times N}=i\tilde \alpha_M \tau_{\ra M}=i\alpha_M \tau_{\ra M} \ .$$ 

Thus $\hat A_M=A_{\inv}$ and  
$$\sigma_{\inv}(M\times N,\F_M\boxtimes \F_N)=[\D^{sign}_{M \times N}(\hat A_M)] \ .$$

Note that
\begin{align*}
\Phi_{M\times N}^{-1}\Gamma_M\Phi_{M\times N} &=-\Gamma_{\ra M}
\end{align*}
and 
$$\Phi_{M\times N}^{-1}\tau_M\Phi_{M\times N}=\Phi_{M\times N}^{-1}\tau_N\Phi_{M\times N}=\tau_N \ .$$

Therefore, in contrast to the closed case, $I=\Gamma_M\tau_M$ commutes neither with $\D_M^{sign}(A_M)$ nor with $(\D_M^{sign} + \Gamma_M\tau_M \D_N^{sign})^{cyl}(\hat A_M)$. This was the motivation for introducing the symmetrized product.

We have that
\begin{align*}
\sigma(M,\F_M) \ten \sigma(N,\F_N)&=[\D_M^{sign}(A_M)] \ten [\D_N^{sign}]\\&=[\sym(\Gamma_M\tau_M,\D_M^{sign}(A_M))] \ten [\D_N^{sign}]\\
&=[\sym\bigl(\Gamma_M\tau_M,(\D_M^{sign} + \Gamma_M\tau_M \D_N^{sign})^{cyl}(\hat A_M)\bigr)]\\
&=[(\D_M^{sign} + \Gamma_M\tau_M \D_N^{sign})^{cyl}(\hat A_M)] \\
&=[\D_{M \times N}^{sign,cyl}(\hat A_M)] \ .
\end{align*}
 
The third equality does not follow directly from Theorem \ref{prod}, but its proof is  analogous.

This concludes the proof of the theorem.
 
\end{proof}

\subsection{The signature class in the odd case}
\label{oddcase}

Now let $M$ be odd-dimensional. Then for $\alpha \in \Lambda^*T^*\ra M \ten \F_M$
$$\tau_M\alpha=idx_1 \wedge \tau_{\ra M}\alpha$$
and $$\tau_M(dx_1 \wedge \alpha)=-i\tau_{\ra M} \alpha \ .$$

Since $\Gamma_M$ anticommutes with $\tau_M$, it induces an isomorphism $\Gamma_M:\Lambda^{\pm}T^*M\ten \F_M \to \Lambda^{\mp}T^*M\ten \F_M$.

The operator $d_M +d_M^*=d_M + \tau_M d_M \tau_M$ commutes with $\tau_M$ and anticommutes with $\Gamma_M$. The (odd twisted) signature operator $d_M^{sign}$ is defined as the restriction of $d_M + \tau_M d_M \tau_M$ to $\Omega^+(M,\F_M)$. 

Then
$$B(d_M^{sign})=i\Phi_M (d_{\ra M}\tau_{\ra M}-\tau_{\ra M}d_{\ra M})\Phi_M^{-1} \ .$$

Define the isometric isomorphism
$$\Pi:\Lambda^{ev}T^*M \ten \F_M \to \Lambda^+T^*M \ten \F_M,~ \alpha \mapsto \frac{1}{\sqrt 2}(\alpha + \tau_M\alpha)$$
Note the connection of
$$\Pi^{-1}d_M^{sign}\Pi=d_M\tau_M+\tau_Md_M$$
with the boundary operator in eq. \ref{boundsign}.

How in turn is the boundary operator of the odd signature operator related to the even signature operator? Consider the isometric isomorphism $$\Xi:\Lambda^* T^*\ra M \ten \F_{\ra M} \to (\Lambda^{ev}T^*M \ten \F_M)|_{\ra M},~\Xi(\alpha):=\frac 12(1+\Gamma_M)(\alpha+ dx_1\wedge\alpha) \ .$$ 
Define $\Sigma_M=\Pi \circ \Xi$. Hence $\Sigma_M(\alpha)=\frac{1}{\sqrt 2}(\alpha + \tau_M\alpha)$ if $\alpha \in \Lambda^{ev} T^*\ra M \ten \F_{\ra M}$ is even, and $\Sigma_M(\alpha)=\frac{1}{\sqrt 2}(dx_1 \wedge \alpha -i \tau_{\ra M}\alpha)$ if $\alpha \in \Lambda^{od} T^*\ra M \ten \F_{\ra M}$.

We have that $$B(d_M^{sign})=\Sigma_Md_{\ra M}^{sign}\Sigma_M^{-1} \ .$$
Furthermore one checks that
$$\gradu_{\ra M}=ic_M(dx_1)=\Sigma_M\tau_{\ra M}\Sigma_M^{-1} \ .$$

For the sake of conformity with \cite{lp6}, we will use this expression for the boundary operator in the following. The following definition is motivated by the boundary conditions considered in \cite[\S 6.4]{lp6}.

\begin{ddd}
\label{definvolod}
Assume given a orthogonal decomposition $\Omega^*_{(2)}(\ra M,\F_{\ra M})=V\oplus W$ with respect to which $\tau_{\ra M},~\Gamma_{\ra M}$ and $\D^{sign}_{\ra M}$ are diagonal. Furthermore assume that $\D^{sign}_{\ra M}|_V$ is invertible. Let $\inv$ be a bounded operator on $\Omega^*_{(2)}(\ra M,\F_{\ra M})$ vanishing on $V$ and whose restriction to $W$ is an involution anticommuting with $\tau_{\ra M}$ and commuting with $\D^{sign}_{\ra M}$ and $\Gamma_{\ra M}$. 

We call a trivializing operator $A$ of $B(d_M^{sign})$ {\rm symmetric} with respect to $\inv$ if it is diagonal with respect to the decomposition $\Sigma_M(V)\oplus \Sigma_M(W)$, vanishes on $\Sigma_M(V)$ and commutes with $\Sigma_M\inv\Sigma_M^{-1}$ on $\Sigma_M(W)$.

If $A$ is a symmetric trivializing operator, then the index class 
$$\sigma_{\inv}(M,\F_M):=[\D_M^{sign}(A)] \in K_1(\A)$$ 
is called the {\rm (twisted) signature class}. 

We call the symmetric trivializing operator $A_{\inv}:=\Sigma_M \inv \Gamma_{\ra M} \Sigma_M^{-1}$ the {\rm canonical} symmetric trivializing operator of $B(d_M^{sign})$ with respect to $\inv$.
\end{ddd}

Note that any symmetric bounded operator that is diagonal with respect to the decomposition $\Sigma_M(V)\oplus \Sigma_M(W)$, vanishes on $\Sigma_M(V)$, anticommutes with $\D^{sign}_{\ra M}$ and commutes with $\inv$ is a symmetric trivializing operator.

As in Lemma \ref{independ} one shows:

\begin{lem}
\label{independgen}
The twisted signature class $\sigma_{\inv}(M,\F_M)$ does not depend on the choice of the symmetric trivializing operator.  
\end{lem}

\begin{proof}
First we outline the general vanishing argument we are using: Consider a selfadjoint Fredholm operator $D$ on a countably generated ungraded Hilbert $C^*$-module $H$. Assume given a unital homomorphism $\rho:C_1 \to B(H)$ such that $\rho(\sigma)$ anticommutes with $D$. We define the even homomorphism $\ov{\rho}:C_1 \to B(H \ten C_1),~\ov{\rho}(\sigma)=\rho(\sigma)\sigma$. Then $\ov{\rho}(\sigma)$ anticommutes with $\sigma D$. Hence 
$$[\sigma D] \in \Ima \bigl(j^*:KK_0(C_1,\A \ten C_1) \to KK_0(\bbbc,\A \ten C_1)\bigr)=0 \ .$$ 
Thus $[D]=0$. (Note that the definition of Kasparov modules for unbounded Fredholm operators in \cite[Def. 2.4]{wa1}, which is the basis for our discussion, contains a sign error: In the odd case, instead of $[D,\rho(b)]$ it should read $(D\rho(b)- (-1)^{\deg b} \rho(b)D)$ for $b$ homogeneous.)

We consider the operator $D:=(\D_Z^{sign} -c(dx_1)(\chi_0 A_0 + \chi_1 A_1))|_{\tilde W}$ on $H:=\tilde W$ defined as in Lemma \ref{independ} with the obvious changes. Let $\rho:C_1 \to B(\tilde W)$ be the unital homomorphism defined by $\rho(\sigma)=\Sigma_M\inv\Sigma_M^{-1}$. Since $\rho(\sigma)$ anticommutes with $c_M(dx_1)=-i\Sigma_M\tau_{\ra M} \Sigma_M^{-1}$, it also anticommutes with $D$. Thus the class $[D] \in KK_1(\bbbc,\A)$ vanishes.     
\end{proof}

\begin{lem}
Let $\Omega^*_{(2)}(\ra M,\F_{\ra M})=V\oplus W$ be an orthogonal decomposition and let $\inv_j,~j=0,1$ be an involution on $W$ such that $\sigma_{\inv_j}(M,\F_M)$ is well-defined. Let $E^+$ be the positive and $E^-$ the negative eigenspace of $\tau_{\ra M}$ on $W$. We identify $E^-$ with $E^+$ using the isomorphism $\inv_0:E^- \to E^+$. There is a unitary $u$ on $E^+$ such that with respect to the decomposition $W=E^+ \oplus E^-$ 
$$\inv_1=\left(\begin{array}{cc} 0 & u^* \\ u & 0 \end{array}\right) \ .$$
Assume that union of the spectra of $u$ and $u^*$ is not equal to $S^1$. Then
$$\sigma_{\inv_0}(M,\F_M)=\sigma_{\inv_1}(M,\F_M) \ .$$
\end{lem}

\begin{proof}
Since $\D^{sign}_{\ra M}$ commutes with $\inv_0$ and anticommutes with $\tau_{\ra M}$, it holds that $$\D^{sign}_{\ra M}=\left(\begin{array}{cc} 0 & D \\ D & 0 \end{array}\right)$$ with $D=(\inv_0\D^{sign}_{\ra M})|_{E^+}$. Furthermore $\D^{sign}_{\ra M}$ also commutes with $\inv_1$. This implies that $Du=u^*D$. Let $C$ be a loop in the intersection of the resolvent sets of $u$ and $u^*$. We assume that $C$ has winding number one with respect to any point in the spectra of $u$ and $u^*$ and that there is a path from the origin to infinity not intersecting the loop. We can define $$a=-i\log(u)=-\frac{1}{2\pi}\int_C \log(\lambda)(u-\lambda)^{-1}d\lambda $$ using any branch of the logarithm.
Then $Da=-aD$. Define $u_t=e^{ita}$ and $\inv_t=\left(\begin{array}{cc} 0 &u_t^* \\ u_t & 0 \end{array}\right)$. We get that $Du_t=u_t^*D$. This in turn implies that $\D_{\ra M}^{sign}$ commutes with $\inv_t$. Analogously $\Gamma_{\ra M}$ commutes with $\inv_t$. Thus the class $\sigma_{\inv_t}(M \times N,\F_M \boxtimes \F_N)$ is well-defined. By homotopy invariance it does not depend on $t$. 
\end{proof}

As in the even case one gets:

\begin{prop}
\label{propodd}
For $j=0,1$ let $\Omega^*_{(2)}(\ra M,\F_{\ra M})=V_j\oplus W_j$ be an orthogonal decomposition and let $\inv_j$ be an involution on $W_j$ such that $\sigma_{\inv_j}(M,\F_M)$ is well-defined.
Assume that $V_0=V_0 \cap V_1 \oplus V_0 \cap W_1$ and
$W_0=W_0\cap V_1 \oplus W_0 \cap W_1$ and that $\inv_0$ and $\inv_1$ restrict to involutions on $W_0 \cap W_1$. Let $\inv_0|_{W_0\cap W_1}$ and $\inv_1|_{W_0\cap W_1}$ fulfill the condition of the previous Lemma. Then
$$\sigma_{\inv_0}(M,\F_M)=\sigma_{\inv_1}(M,\F_M) \ .$$
\end{prop}

The boundary conditions introduced in the following are a special case of those in \cite[\S 6.4]{lp6}. 

Let $m=(\dim M -1)/2$.

Let $V_M$ be the closure of $d^*\Omega^m(\ra M,\F_{\ra M}) \oplus d\Omega^{m-1}(\ra M,\F_{\ra M}) \oplus d^* \Omega^{m+1}(\ra M,\F_{\ra M})\oplus d\Omega^m(\ra M,\F_{\ra M})$ in $\Omega_{(2)}^*(\ra M,\F_{\ra M})$ and let $W_M=V_M^{\perp}$. 

The operators $d,d^*,\tau_{\ra M}$ act on $V_M$ and $W_M$. 

We make the following assumption:

\begin{ass}
\label{assodd}
The closure of $d:\Omega^{m-1}(\ra M,\F_{\ra M}) \to \Omega_{(2)}^{m}(\ra M,\F_{\ra M})$ has closed range. 
\end{ass}

It follows that $\Omega_{(2)}^*(\ra M,\F_{\ra M})=V_M \oplus W_M$ and that $\D^{sign}_{\ra M}$ is invertible on $V_M$.

Let $\cH_{\ra M} \subset W_M$ be the kernel of the Laplacian $\Delta_{\ra M}$ restricted to $\Omega_{(2)}^m(\ra M,\F_{\ra M})$. The Assumption implies that $\cH_{\ra M}$ is a projective $\A$-module. In particular it has an orthogonal complement.

Denote by $\cH_{\ra M}^{\pm}$ the positive resp. negative eigenspace of $\tau_{\ra M}$ restricted to $\cH_{\ra M}$. We also make the following assumption, which is not present in \cite{lp6}. In some of the situations we consider it will be automatically fulfilled. Furthermore it can always be enforced by a stabilization procedure, see \S \ref{stab} for a discussion.

\begin{ass}
\label{assunnec}
The spaces $\cH_{\ra M}^{\pm}$ are isomorphic $\A$-modules.
\end{ass}

This assumption is equivalent to the assumption that there is a submodule $L \subset \cH_{\ra M}$ that is Lagrangian with respect to the skewhermitian form on $\cH_{\ra M}$ induces by $i\tau_{\ra M}$. Let $L^{\perp}$ be its orthogonal complement in $\cH_{\ra M}$. Recall that the definition of a Lagrangian includes the condition $L \oplus L^{\perp}=\cH_{\ra M}$, which is nontrivial for $C^*$-modules.

Let $\Omega^{<}_{\F_{M}}$ be the closed subspace of $W_M$ spanned by forms of degree smaller than $m$, and define $\Omega^{>}_{\F_{M}}$ as the subspace spanned by forms of degree bigger than $m$. 

Let $\alpha_M^L$ be the involution on $W_M$ with positive eigenspace $\Omega^{<}_{\F_M} \oplus L$ and negative eigenspace $\Omega^{>}_{\F_M} \oplus L^{\perp}$. Then $\alpha_M^L$ commutes with $\D^{sign}_{\ra M}$ and $\Gamma_{\ra M}$ and anticommutes with $\tau_{\ra M}$. Thus $\sigma^L(M,\F_M):=\sigma_{\alpha_M^L}(M,\F_M)$ is well-defined. 

If $L_1,L_2\in \cH_{\ra M}$ are two Lagrangians, then there is a difference element $[L_1-L_2] \in K_1(\A)$ and it holds that
$$\sigma^{L_1}(M,\F_M)-\sigma^{L_2}(M,\F_M)=[L_1-L_2] \ .$$
The difference element was described and the statement proven in \cite[\S 6.4]{lp6} using a different definition of odd index classes (via suspension). For the definition used here the result follows from \cite[\S 7-8]{wa1}.  

The difference element vanishes for example if $L_1$ and $L_2$ are homotopic through a path of Lagrangians. 

\section{Product formula for twisted signature classes -- the remaining cases}

In this section we do not make any a priori assumption on the dimensions of $M$ and $N$. We assume that the de Rham operators on $\Omega^*(\ra M,\F_{\ra M})$ and on $\Omega^*(\ra M \times N,\F_{\ra M}\boxtimes \F_N)$ fulfill Assumption \ref{ass} or \ref{assodd}, depending on the dimension of $\ra M$ resp. $\ra M\times N$. 

A warning about gradings: We consider the gradings as they arise in \S \ref{proddir}. In particular vector bundles can only be graded if the underlying manifold is even-dimensional. This implies that the chirality operator $\tau$ need not be a grading operator. Also the grading on the product is as defined in \S \ref{proddir}. 

The proof of the following Lemma is analogous to the proof of Lemma \ref{dirsum}:

\begin{lem}
\label{dirsumgen}
It holds that
\begin{align*}
V_{M\times N}&=(V_M \ten \Omega_{(2)}^*(N,\F_N))\cap V_{M\times N} \oplus  (W_M \ten \Omega_{(2)}^*(N,\F_N))\cap V_{M\times N}\\
W_{M\times N}&=(V_M \ten \Omega_{(2)}^*(N,\F_N))\cap W_{M\times N} \oplus  (W_M \ten \Omega_{(2)}^*(N,\F_N))\cap W_{M\times N} \ .
\end{align*}

If the dimension of $M \times N$ is even, then the operator $\D^{bd}_{\ra M \times N}$ respects the decompositions on the right hand side and is invertible on $V_M \ten \Omega_{(2)}^*(N,\F_N)$.

If the dimension of $M \times N$ is odd, then an analogous statement holds for the operator $\D^{sign}_{\ra M \times N}$.
\end{lem}

The definition of the space $V_N$, which appears in the statement of the following lemma, is the analogue of the definition of $V_M$ for the de Rham operator on $\Omega^*(N,\F_N)$. 

\begin{lem}
\label{hprop}
\begin{enumerate}
\item
If $M\times N$ is odd-dimensional, then 
$$\cH_{\ra M\times N}\subset W_M \ten \Omega_{(2)}^*(N,\F_N) \ .$$
\item
If $M$ is odd-dimensional, then
$$\cH_{\ra M}\ten \Omega_{(2)}^*(N,\F_N)=(\cH_{\ra M}\ten \Omega_{(2)}^*(N,\F_N)) \cap V_{M \times N} \oplus (\cH_{\ra M}\ten \Omega_{(2)}^*(N,\F_N)) \cap W_{M \times N} \ .$$
Furthermore
$$(\cH_{\ra M}\ten \Omega_{(2)}^*(N,\F_N)) \cap V_{M \times N}=\cH_{\ra M} \ten V_N \ .$$ 
In particular
$$V_{M \times N}=(\cH_{\ra M}\ten \Omega_{(2)}^*(N,\F_N)) \cap V_{M \times N} \oplus (\cH_{\ra M}^{\perp}\ten \Omega_{(2)}^*(N,\F_N)) \cap V_{M \times N}\ ,$$
$$W_{M \times N}=(\cH_{\ra M}\ten \Omega_{(2)}^*(N,\F_N)) \cap W_{M \times N} \oplus (\cH_{\ra M}^{\perp}\ten \Omega_{(2)}^*(N,\F_N)) \cap W_{M \times N}\ .$$
\end{enumerate}
\end{lem}

\begin{proof}
1) It is straight-forward to check that $V_M \ten \Omega_{(2)}^*(N,\F_N)$ is orthogonal to $\cH_{\ra M\times N}$.

2) Let $\alpha \in \cH_{\ra M}, ~\beta \in \Omega^k(N,\F_N)$. We only consider the case where $N$ is even-dimensional and $k=\dim N/2$ and leave the other cases to the reader. By the previous Lemma $\alpha \wedge \beta=d\omega_1 +d^*\omega_2+ \omega_3$, where $\omega_1,\omega_2 \in V_{M\times N}$ and $\omega_3 \in W_{M\times N}$. Note that $d^*\omega_1=0, d\omega_2=0$. It follows that $d(\alpha \wedge \beta)=dd^*\omega_2= \Delta \omega_2 \in V_{M\times N}$. Thus $\omega_2=(-1)^{\dim \ra M/2}\Delta^{-1} (\alpha \wedge d\beta)$. It holds that $\Delta_{\ra M}\omega_2=(-1)^{\dim \ra M/2}\Delta^{-1}(\Delta_{\ra M}\alpha\wedge d\beta)=0$. Thus $\omega_2 \in \Ker\Delta_{\ra M}= \cH_{\ra M}\ten \Omega_{(2)}^*(N,\F_N)$. In a similar way one concludes that $\omega_1 \in \cH_{\ra M}\ten \Omega_{(2)}^*(N,\F_N)$. This implies the first equality.
Clearly $d\omega_1 \in \cH_{\ra M} \ten d\Omega^{k-1}(N,\F_N)$ and $d^*\omega_1 \in \cH_{\ra M} \ten d^* \Omega^{k+1}(N,\F_N)$. Thus if $\omega_3=0$, then $\alpha \wedge \beta \in  \cH_{\ra M} \ten V_N$. 

In order to show that $\cH_{\ra M} \ten V_N \subset V_{M\times N}$ it is enough to check that $\cH_{\ra M} \ten V_N$ is orthogonal to $W_{M \times N}$, which is straight-forward.

The last two equations follow from the first in an elementary way.
\end{proof}

Now we prove the product formula in the remaining three cases. The general strategy is as in the proof of Theorem \ref{prodev}.

\subsection{$M$ is even-dimensional and $N$ is odd-dimensional}

We require that Assumption \ref{ass} holds for the de Rham operator on $\Omega^*(\ra M,\F_{\ra M})$ and Assumption \ref{assodd} holds for the de Rham operator on $\Omega^*(\ra M \times N,\F_{\ra M}  \boxtimes \F_N)$.

Lemma \ref{hprop} implies that the involution $\alpha_M\tau_{\ra M} \ten \Gamma_N\tau_N$ restricts to an involution on $\cH_{\ra M \times N}$. Furthermore it anticommutes with $\tau_{\ra M \times N}=-i\tau_{\ra M}\Gamma_{\ra M}\tau_N$.
We define the Lagrangian $L \subset \cH_{\ra M \times N}$ to be its positive eigenspace. (Thus Assumption \ref{assunnec} is fulfilled as well.)

\begin{prop}
\label{prodevod}
It holds that
$$\sigma(M,\F_M) \ten \sigma(N,\F_N)=\sigma^L(M \times N,\F_M \boxtimes \F_N) \in K_1(\A \ten \B) \ .$$
\end{prop}

\begin{proof} 
We have that $$\tau_{M\times N}=\tau_M\Gamma_M \tau_N=\Gamma_M\tau_M \tau_N \ .$$

First assume that $M$ is closed.

By the description of the Kasparov product in \S \ref{produnbound},
$$[\D_M^{sign} + \tau_M\D_N^{sign}]=[\D_M^{sign}] \ten [\D_N^{sign}] \in KK_1(\bbbc,\A \ten \B) \ .$$
The operator $\D_M^{sign} + \tau_M\D_N^{sign}$ acts on $\Omega_{(2)}^*(M,\F_M) \ten \Omega_{(2)}^+(N,\F_N)$. Let 
$$\Theta: \Omega_{(2)}^*(M,\F_M) \ten \Omega_{(2)}^+(N,\F_N)\to \Omega_{(2)}^+(M\times N,\F_M\boxtimes \F_N)$$ be the isomorphism that equals $1 \ten \Gamma_N$ from  $(1-\Gamma_M\tau_M)\Omega^*(M,\F_M) \ten \Omega^+(N,\F_N)$ to $(1-\Gamma_M\tau_M)\Omega^*(M,\F_M) \ten \Omega^-(N,\F_N)$ and the identity on $(1+\Gamma_M\tau_M)\Omega^*(M,\F_M) \ten \Omega^+(N,\F_N)$.

Then $$[\D_M^{sign} + \tau_M\D_N^{sign}]=[\Theta(\D_M^{sign} + \tau_M\D_N^{sign})\Theta^{-1}]=[\D_M^{sign} + \tau_M\Theta \D_N^{sign}\Theta^{-1}] \ .$$
 
For $\alpha \in (1-\Gamma_M\tau_M)\Omega^*(M,\F_M),~\beta \in \Omega^-(N,\F_N)$ 
\begin{align*}
\Theta \D_N^{sign}\Theta^{-1}(\alpha \wedge \beta)&= -\alpha \wedge (d_N+ \tau_N d_N\tau_N)\beta \ .
\end{align*}

Note that the restrictions of $\Gamma_M\tau_M$ and $\tau_N$ to $\Omega^+(M\times N,\F_M\boxtimes \F_N)$ agree. We have that
\begin{align*}
\lefteqn{\tau_M\Theta \D_N^{sign}\Theta^{-1}}\\
&=\frac 12\Gamma_M\tau_N\Bigl((1+\tau_N)(1 \ten (d_N+ \tau_N d_N\tau_N))- (1-\tau_N)(1 \ten (d_N+ \tau_N d_N\tau_N))\Bigr)  \\
&=\Gamma_M\ten (d_N+ \tau_N d_N\tau_N) \ .
\end{align*}

Thus $$\D_{M\times N}^{sign}=\D_M^{sign}+ \tau_M\Theta \D_N^{sign}\Theta^{-1}$$
and therefore
$$[\D_M^{sign}] \ten [\D_N^{sign}]=[\D_{M\times N}^{sign}] \ .$$

Now let $M$ be a manifold with boundary. 

Recall the quantities indexed by $M$, as $\alpha_M,~ V_M,~ W_M$, which were defined in \S \ref{even}. Furthermore $\Psi,\Gamma_2$ are as in \S \ref{evodd}.

Define the involution 
$$\tilde \alpha_M:=(\Sigma_{M\times N}^{-1}\circ\Theta \circ\Psi) (\Gamma_2 \Phi_M\alpha_M\Phi_M^{-1})(\Sigma_{M\times N}^{-1}\circ\Theta \circ \Psi)^{-1}$$ on 
$$\tilde W_M=(\Sigma_{M\times N}^{-1}\circ\Theta \circ \Psi)\bigl((\Phi_M(W_M)\oplus \Phi_M(W_M))\ten \Omega_{(2)}^+(N,\F_N)\bigr) \ .$$ Set 
$$\tilde V_M=(\Sigma_{M\times N}^{-1}\circ\Theta \circ\Psi)\bigl((\Phi_M(V_M) \oplus \Phi_M(V_M))\ten \Omega_{(2)}^+(N,\F_N)\bigr)$$ and let 
$\tilde W_M^{\pm}\subset \tilde W_M$ be the positive resp. negative eigenspace of $\tilde \alpha_M$.

\begin{sublem}
\label{sublevod}
\begin{enumerate}
\item It holds that 
\begin{align*}
\tilde V_M&=V_M \ten \Omega_{(2)}^*(N,\F_N) \\
\tilde W_M&=W_M \ten \Omega_{(2)}^*(N,\F_N) 
\end{align*}
and that $\tilde \alpha_M=\alpha_M\tau_{\ra M} \ten \Gamma_N \tau_N$.
\item The operator $\tilde \alpha_M$ commutes with $\D^{sign}_{\ra M \times N}$ and $\Gamma_{\ra M\times N}$ and anticommutes with $\tau_{\ra M\times N}$ and $\alpha_{M\times N}$.
\end{enumerate}
\end{sublem}
 
\begin{proof}
Let $\alpha_1,\alpha_2 \in \Lambda^*T^*\ra M\ten \F_{\ra M}$ and $\beta \in \Lambda^{ev}T^*N \ten \F_N$.

We have that 
\begin{align*}
\lefteqn{(\Sigma_{M\times N}^{-1}\circ \Theta \circ \Psi)\bigl((\Phi_M(\alpha_1), \Phi_M(\alpha_2)) \wedge (\beta+\tau_N\beta)\bigr)}\\
&=\frac{1}{\sqrt 2}(\Sigma_{M\times N}^{-1}\circ \Theta) \bigl((dx_1 \wedge \alpha_1 + \tau_M(dx_1 \wedge \alpha_1)-i \alpha_2 + i\tau_M(\alpha_2))\wedge (\beta+\tau_N\beta)\bigr) \ .
\end{align*}

Assume now that $\alpha_1,\alpha_2 \in \Lambda^{ev}T^*\ra M\ten \F_{\ra M}$. Then the previous expression equals
$$\frac{1}{\sqrt 2}\Sigma_{M\times N}^{-1}\bigl((dx_1 \wedge \alpha_1 + \tau_M(dx_1 \wedge \alpha_1) - i\alpha_2 + i\tau_M(\alpha_2))\wedge (\beta - \tau_N\beta)\bigr)$$
$$=-(\tau_{\ra M}\alpha_1+\alpha_1)\wedge\tau_N\beta + i(\tau_{\ra M}\alpha_2-\alpha_2)\wedge \beta \ . $$

If $\alpha_1,\alpha_2 \in \Lambda^{od}T^*\ra M\ten \F_{\ra M}$, then it equals
$$\frac{1}{\sqrt 2}\Sigma_{M\times N}^{-1}\bigl((dx_1 \wedge \alpha_1 + \tau_M(dx_1 \wedge \alpha_1) -i\alpha_2 + i\tau_M(\alpha_2))\wedge (\beta+\tau_N\beta)\bigr)$$
$$=(\alpha_1+\tau_{\ra M}\alpha_1)\wedge \beta + i(\tau_{\ra M}\alpha_2-\alpha_2)\wedge \tau_N\beta \ .$$ 

Thus the image of $(\alpha,-i\tau_{\ra M}\alpha)\wedge (\beta+\tau_N\beta)$ under ${\Sigma_{M\times N}^{-1}\circ \Theta \circ \Psi \circ (\Phi_M\oplus \Phi_M)}$ equals $-2\tau_{\ra M}\alpha\wedge\tau_N\beta$ if $\alpha \in \Lambda^{ev}T^*\ra M\ten \F_{\ra M}$, and $2 \alpha \wedge \beta$ if $\alpha \in \Lambda^{od}T^*\ra M\ten \F_{\ra M}$.  

The image of $(\alpha,i\tau_{\ra M}\alpha)\wedge (\beta+\tau_N\beta)$ equals $-2\alpha\wedge\tau_N\beta$ if $\alpha \in \Lambda^{ev}T^*\ra M\ten \F_{\ra M}$, and $2\tau_{\ra M}\alpha \wedge \beta$ if $\alpha \in \Lambda^{od}T^*\ra M\ten \F_{\ra M}$. 

The first part of Sublemma \ref{sublevod} follows.

We define $v_1(\alpha,\beta)$ as the image of $(\alpha, -i\alpha)\wedge (\beta +\tau_N\beta)$ under ${\Sigma_{M\times N}^{-1}\circ \Theta \circ \Psi \circ (\Phi_M\oplus \Phi_M)}$, and $v_2(\alpha,\beta)$ as the image of $(\alpha,i\alpha)\wedge (\beta +\tau_N\beta)$.

The space $\Sigma_{M\times N}^{-1}\tilde W_M^+$ is spanned by the set 
$$\{v_1(\alpha,\beta), ~v_2(\tau_{\ra M}\alpha,\beta)~|~\alpha \in \Omega^{<}_{\F_{M}},~\beta \in \Omega_{(2)}^{ev}(N,\F_N)\} \ .$$ For $\Sigma_{M\times N}^{-1}\tilde W_M^-$ an analogous statement holds with $>$ instead of $<$.

If $\alpha \in \Lambda^{ev}T^*\ra M\ten \F_{\ra M}$, then 
\begin{align*}
v_1(\alpha,\beta)&=-(\tau_{\ra M}\alpha+\alpha)\wedge\tau_N\beta -(\alpha -\tau_{\ra M}\alpha)\wedge \beta\\ 
v_2(\tau_{\ra M}\alpha,\beta)&=(\tau_{\ra M}\alpha+\alpha)\wedge \beta +(\tau_{\ra M}\alpha-\alpha)\wedge \tau_N\beta \ .
\end{align*} 

If $\alpha \in \Lambda^{od}T^*\ra M\ten \F_{\ra M}$, then 
\begin{align*}
v_1(\alpha,\beta)&=(\alpha +\tau_{\ra M}\alpha)\wedge \beta - (\alpha- \tau_{\ra M}\alpha)\wedge \tau_N\beta\\
v_2(\tau_{\ra M}\alpha,\beta)&=-(\alpha+\tau_{\ra M}\alpha)\wedge\tau_N\beta + (\tau_{\ra M}\alpha -\alpha)\wedge \beta \ .
\end{align*}  

Using these equations one checks that $\tilde W_M^{\pm}$ is the positive resp. negative eigenspace of the involution $\alpha_M\tau_{\ra M} \ten \Gamma_N\tau_N$. Thus we get the third equation. From eq. \ref{signprod} it follows that $\D_{\ra M \times N}^{sign}$ commutes with $\tilde \alpha_M$. 
\end{proof}  

We set $\inv:= \tilde \alpha_M$. By the Sublemma $\sigma_{\inv}(M \times N,\F_M\boxtimes \F_N)$ is well-defined. 

The involutions $\alpha_{M\times N}$ and $\tilde \alpha_M$ restrict to involutions on $\tilde W_M \cap W_{M\times N}$. By Lemma \ref{dirsumgen} we can apply Prop. \ref{propodd}, which yields
$$\sigma^L(M \times N,\F_M\boxtimes \F_N)=\sigma_{\inv}(M \times N,\F_M\boxtimes \F_N) \ .$$
 
The canonical symmetric trivializing operator of $B(d_M^{sign})$ with respect to $\alpha_M$ is $A_M=i\Phi_M (\alpha_M \tau_{\ra M})\Phi_M^{-1}$. Then $\hat A_M=\Psi(i \Gamma_1(\Phi_M \alpha_M \tau_{\ra M}\Phi_M^{-1}))\Psi^{-1}$.
Since $\hat A_M$ commutes with $\Psi(\Gamma_2(\Phi_M\alpha_M\Phi_M^{-1}))\Psi^{-1}$, the operator $\Theta \hat A_M \Theta^{-1}$ is a symmetric trivializing operator for $\tilde \alpha_M$. We get that
\begin{align*}
[\D_M^{sign}(A_M)] \ten [\D_N] &= [(\D_M^{sign} + \tau_M\D_N^{sign})^{cyl}(\hat A_M)]\\
&=[\D_{M\times N}^{sign,cyl}(\Theta\hat A_M\Theta^{-1})]\\
&=\sigma_{\inv}(M\times N,\F_M \boxtimes \F_N) \ .
\end{align*}
\end{proof}

\subsection{$M$ is odd-dimensional and $N$ is even-dimensional}

We require that Assumptions \ref{assodd} and \ref{assunnec} hold for the de Rham operator on $\Omega^*(\ra M,\F_{\ra M})$. The de Rham operator on $\Omega^*(\ra M \times N,\F_{\ra M}  \boxtimes \F_N)$ is only required to fulfill Assumption \ref{assodd}. 

The module $\cH_{\ra M \times N}$ decomposes into a direct sum of the projective $\A \ten \B$-modules $$\cH_{\ra M\times N}^{k,l}:=\cH_{\ra M\times N} \cap (\Omega^k(\ra M,\F_{\ra M})\ten \Omega^l(N,\F_N)) \ .$$
The module $\cH_{\ra M\times N}^{k,l}$ is only nontrivial if $k+l=(\dim \ra M + \dim N)/2$.

Let $k=(\dim \ra M)/2,~ l=(\dim \ra N)/2$. Then 
$$\cH_{\ra M\times N}^{k,l} \cong \cH_{\ra M} \ten (\Ker \Delta_N \cap \Omega^{l}(N,\F_N)) \ .$$ Thus any Lagrangian $L \in \cH_{\ra M}$ defines a Lagrangian in $\cH_{\ra M \times N}^{k,l}$. From this and Lemma \ref{hprop} (1) it follows that the involution $\alpha_M^L \ten \Gamma_N$
 restricts to an involution on $\cH_{\ra M \times N}$.

Define the Lagrangian 
$L_{\ten} \subset \cH_{\ra M\times N}$ as the positive eigenspace of the involution $\alpha_M^L \ten \Gamma_N$ restricted to $\cH_{\ra M\times N}$.
(Note that the existence of this Lagrangian implies that Assumption \ref{assunnec} is fulfilled.)

By construction 
$$\alpha_{M\times N}^{L_{\ten}}|_{\cH_{\ra M\times N}}=\alpha_M^L \ten \Gamma_N|_{\cH_{\ra M\times N}} \ .$$

\begin{prop}
\label{prododev}
It holds that
$$\sigma^L(M,\F_M) \ten \sigma(N,\F_N)=\sigma^{L_{\ten}}(M \times N,\F_M \boxtimes \F_N) \in K_1(\A \ten \B) \ .$$
\end{prop}

\begin{proof}
We have that $\tau_{M \times N}=\tau_M \tau_N$.

First assume that $M$ is closed. 
 
By \S \ref{produnbound}
$$[\tau_N\D_M^{sign} + \D_N^{sign}]=[\D_M^{sign}] \ten [\D_N^{sign}] \in KK_1(\bbbc,\A \ten \B) \ .$$
The operator $\tau_N\D_M^{sign} + \D_N^{sign}$ acts on $\Omega_{(2)}^+(M,\F_M) \ten \Omega_{(2)}^*(N,\F_N)$. 

Let
$$\Theta: \Omega_{(2)}^+(M,\F_M) \ten \Omega_{(2)}^*(N,\F_N)\to \Omega_{(2)}^+(M\times N,\F_M\boxtimes \F_N)$$
be the isomorphism that equals $\Gamma_M$ from  $\Omega^+(M,\F_M) \ten \Omega^-(N,\F_N)$ to $\Omega^-(M,\F_M) \ten \Omega^-(N,\F_N)$ and the identity on $\Omega^+(M,\F_M) \ten \Omega^+(N,\F_N)$.
Note that $\Gamma_M \D^{sign}_N=\Theta \D_N^{sign} \Theta^{-1}$.

The signature operator on $\Omega_{(2)}^+(M\times N,\F_M \boxtimes \F_N)$ fulfills
$$\D^{sign}_{M \times N}=\tau_M\Theta \D^{sign}_M \Theta^{-1} + \Gamma_M \D^{sign}_N=\Theta (\tau_N\D_M^{sign} + \D_N^{sign})\Theta^{-1}\ . $$

Thus
$$[\tau_N\D_M^{sign} + \D_N^{sign}]=[\D^{sign}_{M \times N}] \ .$$

Now let $M$ be a manifold with boundary.

Define the involution 
$$\tilde \alpha_M^{L}=(\Sigma_{M\times N}^{-1}\circ\Theta \circ \Sigma_M)(\alpha_M^L\ten \tau_N)(\Sigma_{M\times N}^{-1}\circ\Theta\circ \Sigma_M)^{-1}$$
on $$\tilde W_M=(\Sigma_{M\times N}^{-1}\circ\Theta\circ \Sigma_M)(W_M \ten \Omega_{(2)}^*(N,\F_N))$$
and set 
$$\tilde V_M=(\Sigma_{M\times N}^{-1}\ten \Theta \circ \Sigma_M)(V_M \ten \Omega_{(2)}^*(N,\F_N)) \ .$$
Furthermore let $\tilde W_M^{\pm}$ be the positive resp. negative eigenspace of $\tilde \alpha_M^L$.

Compare the following sublemma with Sublemma \ref{sublevod}.
\begin{sublem}
\begin{enumerate}
\item It holds that 
\begin{align*}
\tilde V_M&=V_M \ten \Omega_{(2)}^*(N,\F_N) \\
\tilde W_M&=W_M \ten \Omega_{(2)}^*(N,\F_N) 
\end{align*}
and $$\tilde \alpha_M^L= \alpha_M^L \ten \Gamma_N \ .$$
\item The operator $\tilde \alpha_M^L$ commutes with $\Gamma_{\ra M\times N},\alpha_{M\times N}^{L_{\ten}}$ and $\D^{sign}_{\ra M \times N}$  and anticommutes with $\tau_{\ra M\times N}$.
\end{enumerate} 
\end{sublem}
 
\begin{proof}

For $\alpha \in \Omega^{ev}(\ra M,\F_M),~\beta \in \Omega^*(N,\F_N)$
\begin{align*}
(\Sigma_{M\times N}^{-1}\circ \Theta)(\Sigma_M(\alpha) \wedge (\beta \pm \tau_N\beta))&=\frac{1}{\sqrt 2} (\Sigma_{M\times N}^{-1}\circ \Theta)((\alpha+ \tau_M \alpha)\wedge (\beta \pm \tau_N \beta))\\
&=\frac{1}{\sqrt 2}\Sigma_{M\times N}^{-1}((\alpha \pm \tau_M \alpha)\wedge (\beta \pm\tau_N \beta)) \ .
\end{align*}

For $\alpha \in \Omega^{od}(\ra M,\F_M),~\beta \in \Omega^*(N,\F_N)$
\begin{align*}
(\Sigma_{M\times N}^{-1}\circ \Theta)(\Sigma_M(\alpha) \wedge (\beta \pm \tau_N\beta))&=\frac{1}{\sqrt 2} (\Sigma_{M\times N}^{-1}\circ \Theta)((dx_1 \wedge \alpha - i\tau_{\ra M} \alpha)\wedge (\beta \pm \tau_N \beta))\\
&=\frac{1}{\sqrt 2}\Sigma_{M\times N}^{-1}((dx_1 \wedge \alpha \mp i \tau_{\ra M}\alpha)\wedge (\beta \pm \tau_N \beta)) \ .
\end{align*}
In both cases this equals $\alpha \wedge (\beta \pm \tau_N\beta)$ if $\beta \in \Omega^{ev}(N,\F_N)$ and
$\pm i\tau_{\ra M}\alpha \wedge (\beta \pm \tau_N\beta)$ if $\beta \in \Omega^{od}(N,\F_N)$. (These statements hold true if we choose the sign above resp. below everywhere.)

Thus
\begin{align*}
\tilde W_M^+&=(\Omega^{<}_{\F_{M}}\oplus L)\ten \Omega_{(2)}^{ev}(N,\F_N) \oplus  (\Omega^{>}_{\F_{M}}\oplus L^{\perp})\ten \Omega_{(2)}^{od}(N,\F_N)\\
\tilde W_M^-&=(\Omega^{<}_{\F_{M}}\oplus L)\ten \Omega_{(2)}^{od}(N,\F_N) \oplus  (\Omega^{>}_{\F_{M}}\oplus L^{\perp})\ten \Omega_{(2)}^{ev}(N,\F_N) \ .
\end{align*}

It follows that $\tilde W_M^{\pm}$ is the positive resp. negative eigenspace of the involution $\alpha_M^L\ten \Gamma_N$. Eq. \ref{signprod} implies that the involution commutes with the signature operator on the boundary $\D_{\ra M \times N}^{sign}$. It clearly commutes with $\alpha_{M\times N}^{L_{\ten}}$ and anticommutes with $\tau_{\ra M \times N}=\tau_{\ra M}\tau_N$.
\end{proof}

We write $\inv=\tilde \alpha_M^L$. By the Sublemma $\sigma_{\inv}(M\times N,\F_M\boxtimes \F_N)$ is well-defined. Using Lemma \ref{dirsumgen} and Lemma \ref{hprop} one checks that $\alpha_{M\times N}$ and $\tilde \alpha_M^L$ restrict to involutions on $\tilde W_M \cap W_{M\times N}$. By Prop. \ref{propodd} 
$$\sigma_{\inv}(M\times N,\F_M\boxtimes \F_N)=\sigma(M\times N,\F_M\boxtimes \F_N) \ .$$

Let $A_M$ be the canonical symmetric trivializing operator for $B(d_M^{sign})$ with respect to $\alpha_M^L$. Then $\hat A_M=\Sigma_M(\alpha_M^L \Gamma_{\ra M})\Sigma_M^{-1}$. From the calculations in the proof of the Sublemma it also follows that 
$$(\Sigma_{M\times N}^{-1}\ten \Theta \circ \Sigma_M)\Gamma_{\ra M}\tau_N (\Sigma_{M\times N}^{-1}\ten \Theta \circ \Sigma_M)^{-1}=\Gamma_{\ra M}\tau_N \ .$$
Hence
$$(\Sigma_{M\times N}^{-1}\circ \Theta)\hat A_M (\Sigma_{M\times N}^{-1}\circ \Theta)^{-1}=\tilde \alpha_M^L\Gamma_{\ra M}\tau_N \ .$$ Since $\tilde \alpha_M^L\Gamma_{\ra M}\tau_N$ anticommutes with $\D^{sign}_{\ra M \times N}$ and commutes with $\tilde \alpha_M^L$, the operator $\Theta\hat A_M \Theta^{-1}$ is a symmetric trivializing operator of $B(d_{M\times N}^{sign})$ with respect to $\inv=\tilde \alpha_M^L$.

Thus
$$\sigma_{\inv}(M\times N,\F_M\boxtimes \F_N)=[\D^{sign}_{M \times N}(\tau_N\Theta\hat A_M \Theta^{-1})] \ .$$

Arguing as in the closed case we have that 
\begin{align*}
\sigma^L(M,\F_M) \ten \sigma(N,\F_N)&=[\D_M^{sign}(A_M)] \ten [\D_N^{sign}]\\
&=[(\tau_N\D_M^{sign}+\D_N^{sign})^{cyl}(\hat A_M)]\\
&=[(\D_{M \times N}^{sign})^{cyl}(\Theta \hat A_M\Theta^{-1})] \ .
\end{align*}
\end{proof}

\subsection{$M,N$ are odd-dimensional}

Let Assumptions \ref{assodd} and \ref{assunnec} hold for the de Rham operator on $\Omega^*(\ra M,\F_{\ra M})$ and Assumption \ref{ass} for the de Rham operator on $\Omega^*(\ra M \times N,\F_{\ra M}  \boxtimes \F_N)$. Let $L \subset \cH_{\ra M}$ be a Lagrangian.

\begin{prop}
\label{prododod}

It holds that
$$2\sigma^L(M,\F_M) \ten \sigma(N,\F_N)=\sigma(M \times N,\F_M \boxtimes \F_N) \in K_0(\A \ten \B) \ .$$
\end{prop}

\begin{proof}
We have that $$\tau_{M\times N}=-i\tau_M\Gamma_M\tau_N \ .$$

First let $M$ be closed. In the following we will denote by $\D_M^{sign}|_X$ the closure of $d_M + \tau_M d_M \tau_M$ acting on $X \subset \Omega_{(2)}^*(M,\F_M)$. Without specification $\D_M^{sign}$ is understood to act on the space $\Omega_{(2)}^+(M,\F_M)$, as before. The same applies to $\D_N^{sign}$.  

We may identify $\Omega_{(2)}^+(M,\F_M) \oplus \Omega_{(2)}^+(M,\F_M)$ with $\Omega_{(2)}^*(M,\F_M)$ by applying the isomorphism $\Gamma_M:\Omega_{(2)}^+(M,\F_M) \to \Omega_{(2)}^-(M,\F_M)$ to the second summand. 

From \S \ref{oddodd} we get that $\Gamma_1=\tau_M, \Gamma_2=i\tau_M\Gamma_M$ and then from \S \ref{produnbound}
$$[\D_M^{sign}] \ten [\D_N^{sign}]=[(\D^{sign}_M + i \tau_M\Gamma_M\D_N^{sign})|_{\Omega^*_{(2)}(M,\F_M) \ten \Omega_{(2)}^+(N,\F_N)}] \ .$$
Here the grading operator on $\Omega^*_{(2)}(M,\F_M) \ten \Omega_{(2)}^+(N,\F_N)$ is $-i\Gamma_1\Gamma_2=\Gamma_M$.

Define the isometric isomorphism
$$\Theta:\bigl(\Omega^*_{(2)}(M,\F_M) \ten \Omega_{(2)}^+(N,\F_N)\bigr)^2 \to \Omega_{(2)}^*(M\times N, \F_M \boxtimes \F_N) \ ,$$ 
$$\Theta(\omega_1,\omega_2)=\omega_1 + \Gamma_N\omega_2 \ .$$
On $\Omega_{(2)}^*(M\times N, \F_M \boxtimes \F_N)$ 
\begin{align*}
\Theta(\D^{sign}_M + i \tau_M\Gamma_M\D_N^{sign})\Theta^{-1}&=\D^{sign}_M + i \tau_M\Gamma_M\tau_N\D_N^{sign} \\
&=\D^{sign}_M -  \tau_{M\times N}\D_N^{sign} \ .
\end{align*}  
Hence
$$2[\D_M^{sign}] \ten [\D_N^{sign}]=[(\D^{sign}_M -  \tau_{M\times N}\D_N^{sign})|_{\Omega_{(2)}^*(M\times N, \F_M \boxtimes \F_N)}] \ .$$

In order to compare the latter class with the signature class we define a unitary operator
${\mathcal Z}$ on $\Omega_{(2)}^*(M\times N, \F_M \boxtimes \F_N)$ by 
$${\mathcal Z}(\alpha \wedge \beta)=\frac{1}{\sqrt 2}(\alpha \wedge \beta + \tau_{M\times N}(\Gamma_M\alpha \wedge \beta))$$ for $\alpha \in \Omega^*_{(2)}(M,\F_M),~\beta \in \Omega^*_{(2)}(N,\F_N)$.
Then
$${\mathcal Z}\Gamma_M{\mathcal Z}^{-1}=\tau_{M\times N} \ .$$ 
Furthermore for $\alpha \in \Omega^{ev}(M,\F_M), \beta \in \Omega^*(N,\F_N)$
\begin{align*}
\lefteqn{\frac{1}{\sqrt 2}{\mathcal Z}(d^{sign}_M -  \tau_{M\times N}d_N^{sign}){\mathcal Z}^{-1}(\alpha \wedge \beta + \tau_{M \times N}(\alpha \wedge \beta))}\\
&= {\mathcal Z}\bigl((d^{sign}_M - \tau_{M\times N}d_N^{sign})(\alpha \wedge \beta)\bigr)\\
&=\frac{1}{\sqrt 2}\bigl(d^{sign}_M \alpha \wedge \beta - \tau_{M\times N}(d_M^{sign}\alpha \wedge \beta)+ d_N^{sign}(\alpha \wedge \beta)-  \tau_{M\times N}d_N^{sign}(\alpha \wedge \beta)\bigr)\\
&=\frac{1}{\sqrt 2}\bigl(d^{sign}_M \alpha \wedge \beta - i d_M^{sign}(\tau_M\alpha \wedge \tau_N\beta) + d_N^{sign}(\alpha \wedge \beta) + i \tau_M\alpha \wedge d_N^{sign}\tau_N\beta\bigr)\\
&=\frac{1}{\sqrt 2}(d_M^{sign} + \Gamma_M d_N^{sign})\bigl(\alpha \wedge \beta+ \tau_{M \times N}(\alpha \wedge \beta)\bigr)\\
&=\frac{1}{\sqrt 2}d_{M\times N}^{sign}(\alpha \wedge \beta+ \tau_{M \times N}(\alpha \wedge \beta)) \ .
\end{align*}

The last equation follows from eq. \ref{signprod}. It follows that $$d_{M\times N}^{sign}={\mathcal Z}(d^{sign}_M -  \tau_{M\times N}d_N^{sign}){\mathcal Z}^{-1}$$ as an operator from $\Omega^+(M\times N, \F_M \boxtimes \F_N)$ to $\Omega^-(M\times N, \F_M \boxtimes \F_N)$. Since both sides of the equation are essentially selfadjoint, the equation holds on $\Omega^*(M\times N, \F_M \boxtimes \F_N)$. Hence in the closed case
$$2[\D_M^{sign}] \ten [\D_N^{sign}]=[\D_{M\times N}^{sign}] \in K_0(\A \ten \B) \ .$$

Now let $M$ be a manifold with boundary. The isomorphism $\Psi$ defined in \S \ref{oddodd} is here a map from $(\Lambda^+T^*M\ten \F_M)|_{\ra M} \boxtimes (\Lambda^+T^*N\ten \F_N)$ to $(\Lambda^{ev}T^*M\ten \F_M)|_{\ra M} \boxtimes (\Lambda^+T^*N\ten \F_N)$ given by 
$$\Psi(\omega)=\frac{1}{\sqrt 2}(\omega +\Gamma_M\omega) \ .$$

We define the isomorphism $\ov{\Psi}$ from $\bigl((\Lambda^*T^*\ra M \ten \F_{\ra M})\boxtimes (\Lambda^+T^*N \ten\F_N)\bigr)^2$  to $\bigl(\Lambda^*T^*(M\times N)\ten (\F_M\boxtimes \F_N)\bigr)|_{\ra M\times N}$ by 
\begin{align*}
\ov{\Psi}(\omega_1,\omega_2)&=\Theta\bigl((\Psi\circ\Sigma_M)(\omega_1),(\Psi\circ \Sigma_M)(\omega_2)\bigr) \\
&=\frac{1}{\sqrt 2}\bigl(\Sigma_M(\omega_1)+\Gamma_M\Sigma_M(\omega_1)+ \Gamma_N\Sigma_M(\omega_2)+ \Gamma_{M\times N}\Sigma_M(\omega_2)\bigr) \ .
\end{align*}
We set 
$$\tilde W_M=(\Phi_{M\times N}^{-1} \circ {\mathcal Z}\circ \ov{\Psi})\bigl((W_M \oplus W_M) \ten \Omega_{(2)}^+(N,\F_N)\bigr) \ ,$$ 
$$\tilde V_M=(\Phi_{M\times N}^{-1} \circ {\mathcal Z}\circ \ov{\Psi})\bigl((V_M \oplus V_M)\ten \Omega_{(2)}^+(N,\F_N)\bigr)$$ 
and
$$\tilde \alpha_M^L=(\Phi_{M\times N}^{-1} \circ {\mathcal Z}\circ \ov{\Psi})(\alpha_M^L\Gamma_{\ra M} \oplus \alpha_M^{L}\Gamma_{\ra M})(\Phi_{M\times N}^{-1} \circ {\mathcal Z}\circ \ov{\Psi})^{-1} \ .$$

The following sublemma is similar to Sublemma \ref{sublevev}.
\begin{sublem}
\begin{enumerate}
\item It holds that 
\begin{align*}
\tilde V_M&=V_M \ten \Omega_{(2)}^*(N,\F_N) \\
\tilde W_M&=W_M \ten \Omega_{(2)}^*(N,\F_N)
\end{align*}
an that $\tilde \alpha_M^L=-\alpha_M^L$.
\item The operator $\tilde \alpha_M^L$ anticommutes with $\D^{bd}_{\ra M\times N}$ and $\tau_{\ra M\times N}$ and  commutes with $\alpha_{M\times N}$.
\end{enumerate} 
\end{sublem}

\begin{proof}

Let $\alpha_1,\alpha_2 \in \Omega^{ev}(\ra M,\F_M),~ \beta_1,\beta_2 \in \Omega^+(N,\F_N)$ and set  $\omega:=\alpha_1 \wedge \beta_1+\alpha_2 \wedge \Gamma_N \beta_2$. We have that
\begin{align*}
\lefteqn{(\Phi_{M\times N}^{-1} \circ {\mathcal Z}\circ \ov{\Psi})(\alpha_1 \wedge \beta_1,\alpha_2 \wedge\beta_2)} \ .\\
&=(\Phi_{M\times N}^{-1} \circ {\mathcal Z})(\omega)\\
&=\frac{1}{\sqrt 2} \Phi_{M\times N}^{-1}(\omega + \tau_{M\times N}(\omega))\\
&=\tau_{\ra M\times N}(\omega) \ .
\end{align*}

For $(\dim \ra M)/2$ even $L \subset \Omega^{ev}(\ra M,\F_M)$. If $\alpha_1,~\alpha_2 \in L$, then $\tau_{\ra M\times N}(\omega) \in L^{\perp}\ten \Omega^*(N,\F_N)$. 

For $\alpha_1,\alpha_2 \in \Omega^{od}(\ra M,\F_M),~ \beta_1,\beta_2 \in \Omega^+(N,\F_N)$ and $\omega$ as before we have that
\begin{align*}
\lefteqn{(\Phi_{M\times N}^{-1} \circ {\mathcal Z}\circ \ov{\Psi})(\alpha_1 \wedge \beta_1, \alpha_2 \wedge\beta_2)} \ .\\
&=(\Phi_{M\times N}^{-1} \circ {\mathcal Z})(dx_1\wedge \omega)\\
&=\frac{1}{\sqrt 2} \Phi_{M\times N}^{-1}(dx_1\wedge\omega + \tau_{M\times N}(dx_1 \wedge \omega))\\
&=\omega \ .
\end{align*}

If $(\dim \ra M)/2$ is odd and $\alpha_1,~\alpha_2 \in L \subset \Omega^{od}(\ra M,\F_M)$, then $\omega \in L\ten \Omega^*(N,\F_N)$. 

From this one deduces (1).
Hence $\tilde \alpha_M^L$ anticommutes with $\tau_{\ra M \times N}=\Gamma_{\ra M}\tau_{\ra M}\tau_N$ and commutes with $d_{\ra M\times N}$. By Lemma \ref{hprop} the operator $\alpha_{M \times N}$ is diagonal with respect to the decomposition $\cH_{\ra M} \ten \Omega_{(2)}^*(N,\F_N) \oplus \cH_{\ra M}^{\perp} \ten \Omega_{(2)}^*(N,\F_N)$. Using this one gets (2).
\end{proof}

We set $\inv=\tilde \alpha_M^L$. By the Sublemma $\sigma_{\inv}(M\times N,\F_M\boxtimes \F_N)$ is well-defined. Lemma \ref{dirsumgen} and Lemma \ref{hprop} imply that $\alpha_{M\times N}$ and $\tilde \alpha_M^L$ restrict to involutions on $\tilde W_M \cap W_{M\times N}$. 

By Prop. \ref{propev} we get that
$$\sigma_{\inv}(M\times N,\F_M\boxtimes \F_N)=\sigma(M\times N,\F_M\boxtimes \F_N) \ .$$

Let $A_{M}^L$ be the canonical symmetric trivializing operator for $B(d_M^{sign})$ with respect to $\alpha_M^L$. 

By definition $\hat A_M^{L}=(\Psi \circ \Sigma_M)(\alpha_M^L \Gamma_{\ra M})(\Psi \circ \Sigma_M)^{-1}$. Hence
$$(\Phi_{M\times N}^{-1}\circ {\mathcal Z}\circ \Theta)(\hat A_M^L\oplus \hat A_M^L)(\Phi_{M\times N}^{-1}\circ {\mathcal Z}\circ \Theta)=\tilde \alpha_M^L \ .$$ 
Thus 
\begin{align*}
2[\D_M^{sign}(A_M^{L})] \ten [\D_N^{sign}]&=[\D_{M\times N}^{sign}\bigl(({\mathcal Z}\circ \Theta)(\hat A_M^L\oplus \hat A_M^L)({\mathcal Z}\circ \Theta)^{-1}\bigr)]\\
&=[\D_{M\times N}^{sign}(\Phi_{M\times N} \tilde \alpha_M^L \Phi_{M\times N}^{-1})]\\
&=\sigma_{\inv^{opp}}(M\times N,\F_M \boxtimes \F_N)\\
&=\sigma_{\inv}(M\times N,\F_M \boxtimes \F_N) \ .
\end{align*}
The first equation follows from the product formula in \S \ref{oddodd}. The last equation follows from eq. \ref{eqinvinvopp}.
\end{proof}

\section{Product formula for higher signatures}
\label{highsig}

In the following we give a slight generalization of the previous product formulas, which also applies to higher signatures. 

Let $\Ca$ be unital $C^*$-algebra and let $\varphi:\A \ten \B \to \Ca$ be a unital $C^*$-homomorphism. There is an induced map $\varphi_*:K_*(\A \ten \B) \to K_*(\Ca)$. 

The bundle $(\F_M \boxtimes \F_N) \ten_{\varphi}\Ca$ is a flat $\Ca$-vector bundle on $M \times N$. 

The proof of the following theorem is nearly literally as before, if at the right places one plugs in tensor products $\ten_{\varphi}\Ca$. Also as before, Assumption \ref{assunnec} in the statement of the theorem will be automatically fulfilled for the de Rham operator on $\Omega^*(\ra M \times N,(\F_{\ra M} \boxtimes \F_N) \ten_{\varphi}\Ca)$ if $M \times N$ is odd-dimensional.

\begin{theorem}
In the following we assume that the respective assumptions (i.e. Assumption \ref{ass} resp. Assumptions \ref{assodd} and \ref{assunnec}, depending on the dimensions of $M$ and $N$) hold for the de Rham operators on $\Omega^*(\ra M,\F_{\ra M})$ and on $\Omega^*(\ra M \times N,(\F_{\ra M} \boxtimes \F_N) \ten_{\varphi}\Ca)$. 

\begin{enumerate}
\item If $M$ and $N$ are even-dimensional, then
$$\varphi_*\bigl(\sigma(M,\F_M) \ten \sigma(N,\F_N)\bigr)=\sigma(M\times N,(\F_M \boxtimes \F_N) \ten_{\varphi}\Ca) \ .$$
\item If $M$ is even-dimensional and $N$ is odd-dimensional and $L$ is the positive eigenspace of the involution $\alpha_M\tau_{\ra M} \ten \Gamma_N\tau_N$ on $\cH_{\ra M \times N}\subset \Omega^*(\ra M \times N,(\F_{\ra M} \boxtimes \F_N) \ten_{\varphi}\Ca)$, then  
$$\varphi_*\bigl(\sigma(M,\F_M) \ten \sigma(N,\F_N)\bigr)=\sigma^L(M\times N,(\F_M \boxtimes \F_N) \ten_{\varphi}\Ca) \ .$$
\item  If $M$ is odd-dimensional and $N$ is even-dimensional and $L\subset \cH_{\ra M}$ is a Lagrangian, then  we can define $L_{\ten} \subset \cH_{\ra M \times N}$ as before and get
$$\varphi_*\bigl(\sigma^L(M,\F_M) \ten \sigma(N,\F_N)\bigr)=\sigma^{L_{\ten}}(M\times N,(\F_M \boxtimes \F_N) \ten_{\varphi}\Ca) \ .$$
\item
If $M$ and $N$ are odd-dimensional and $L \subset \cH_{\ra M}$ is a Lagrangian, then
$$2\varphi_*\bigl(\sigma^L(M,\F_M) \ten \sigma(N,\F_N)\bigr)=\sigma(M\times N,(\F_M \boxtimes \F_N) \ten_{\varphi}\Ca) \ .$$
\end{enumerate}
\end{theorem}

This result applies to higher signatures: 

Let $\tilde M$ resp. $\tilde N$ be a Galois covering of $M$ resp. $N$ and let  $\pi_M$ resp. $\pi_N$ be the group of deck transformations. By definition $\Pj_M=\tilde M \times_{\pi_M} C^*_r\pi_M$ is the associated Mishenko-Fomenko bundle and $\sigma(M,\Pj_M)$ is the higher signature class of $M$ associated to the covering. 

The group $\pi_{M\times N}:=\pi_M \times \pi_N$ is the decktransformation group with respect to the covering $\tilde M \times \tilde N \to M \times N$. Let $\Pj_{M\times N}$ be the corresponding Mishenko-Fomenko bundle. There is a canonical unital $C^*$-homomorphism $$\varphi:C_r^*\pi_M \ten C_r^* \pi_N \to C_r^*(\pi_{M \times N}) \ ,$$ 
and it holds that 
$$\Pj_{M\times N}=(\Pj_M\boxtimes \Pj_N)\ten_{\varphi} C_r^*(\pi_{M \times N}) \ .$$ 

Thus from the previous proposition one gets a product formula for higher signature classes. In this case (see \cite[Lemma 3.1]{llk}) Assumption \ref{ass} is equivalent to the $m$-th Novikov Shubin invariant $\alpha_m(\ra \tilde M)$ being $\infty^+$, whereas Assumption \ref{assodd} is equivalent to $\alpha_m(\ra \tilde M)=\alpha_{m+1}(\ra \tilde M)=\infty^+$. If the $m$-th Betti number $b_m(\ra \tilde M)$ vanishes, then $\cH_{\ra M}=0$, thus Assumption \ref{assunnec} is fulfilled. 
For products these conditions can be checked by using the product formulas for Novikov-Shubin invariants \cite[Theorem 2.55(3)]{l} and $L^2$-Betti numbers \cite[Theorem 1.35(4)]{l}. Examples for which the conditions are fulfilled can be found in \cite[p.~563]{llp}. 

In a similar way the product formula applies to twisted higher signatures as studied in \cite{lptwisthigh}.
In \cite[\S 2]{lptwisthigh} examples were given where the Laplacian on $\Omega^*(\ra M,\F_{\ra M} \ten_{\varphi}\Ca)$ is invertible. This implies that also the Laplacian on $\Omega^*(\ra M \times N,(\F_{\ra M} \boxtimes \F_N) \ten_{\varphi}\Ca)$ is invertible, thus the conditions of the theorem are fulfilled. 

Product formulas for geometric invariants are relevant for the following question: Assume that $M_1, M_2$ are non-isomorphic elements in a suitable category (topological spaces up to homotopy/homeomorphism, manifolds up to diffeomorphism, manifolds with boundary up to homotopy/homeomorphism/diffeomorphism etc.). Under which conditions on a closed manifold $N$ does it follow that are $M_1 \times N$, $M_2 \times N$ not isomorphic? See the motivating examples for the definition of the higher $\rho$-invariants given in \cite{we}.   

By applying the homotopy invariance result of \cite{llp} (which was proven there using different boundary conditions; see the end of \S \ref{normalization} for the justification of using it here)  we obtain the following corollary, which for simplicity we only formulate in the even-dimensional case and only for universal coverings:

\begin{cor}
Let $M_1, M_2$ be orientable even-dimensional manifolds with boundary having the same fundamental group $\pi_M$. Let $N$ be an orientable even-dimensional closed manifold with fundamental group $\pi_N$. Assume that the higher signature classes of $M_1, M_2, M_1 \times N, M_2 \times N$ are well-defined (with respect to the universal coverings). 

If the higher signature classes of $M_1, M_2$ do not agree up to sign in $K_0(C_r^*\pi_M) \ten \bbbq$ and the higher signature class of $N$ does not vanish in $K_0(C_r^*\pi_N) \ten \bbbq$, then $M_1 \times N$ is not homotopic to $M_2 \times N$ as a manifold with boundary. 
\end{cor}

The non-vanishing of higher signature classes for manifolds with boundary can be proven by using the higher Atiyah-Patodi-Singer index theorem of Leichtnam-Piazza, see \cite{llp} and references therein.

The example in \cite[p.~624\,f.]{llp} illustrates the corollary. While no detailed argument was given there, for the calculation of the relevant higher signatures a product formula for Chern characters and $\eta$-forms might have been used. Alternatively one may use the above product formula.

\section{Further remarks}

\subsection{Stabilization in the odd case}
\label{stab} 

Let $\dim M=2m+1$.

We sketch the stabilization trick and derive product formulas if Assumption \ref{assodd} holds for the de Rham operator on $\Omega^*(\ra M,\F_{\ra M})$ but Assumption \ref{assunnec} does not hold. The stabilization comes at a price: We need to require that Assumption \ref{ass} holds for the de Rham operator on $\Omega^*(N,\F_N)$ if $N$ is odd-dimensional and Assumption \ref{assodd} if $N$ is even-dimensional. If $N$ is odd-dimensional, it follows that the signature class $\sigma(N,\F_N)$ vanishes. This is already suggested by the product formula in Prop. \ref{prododod}, where the left hand side depends on a Lagrangian while the right hand side is not. If $N$ is even-dimensional, it follows that $\cH_N:=\Ker \Delta_N \cap \Omega^{(\dim N)/2}(N,\F_N)$ is a projective $\B$-module.

The construction relies on the concept of ``stable'' Lagrangians \cite[\S 3]{llp}. While clearly inspired by it, our stabilization procedure differs from the one in \cite{lp6} and avoids the additional choice of a submodule as specified in \cite[Prop. 11]{lp6}.

Let $X$ be an odd-dimensional manifold with boundary and assume that the middle degree homology $H$ of $\ra X$ is non-zero. Let $\F_X=\A^k$. Then $\cH_{\ra X}=H \ten \A^k$. If $L_0 \in H$ is a Lagrangian, then $\sigma^{L_0\ten\A^k}(X,\F_X)=0$. For the following choose a trivialization $H=\bbbc^{2n}$, where $\bbbc^{2n}$ is endowed with the standard skewhermitian form (which is induced by the standard symplectic form on $\bbbr^{2n}$).

Consider the disjoint union $M \cup X$. Let $k$ be large enough such that there is a Lagrangian $L \subset \cH_{\ra M} \oplus \A^{2nk}$. The existence follow from \cite[Lemma 3.4]{llp} since $[\D^{sign}_{\ra M}]=0$ by the bordism invariance of the index. We define
$$\sigma^L(M,\F_M):=\sigma^L(M\cup X,\F_M \cup \F_X) \ .$$

It is often useful to choose $X$ with dimension different from $M$, see below. Thus the right hand side requires a straightforward extension of the definition of the signature class to accommodate for components of differing dimension.
 
One checks that the definition makes sense. It does not depend on the choice of $X$ nor of the trivialization $H\cong \bbbc^{2n}$, only on the choice of $L$ in $\cH_{\ra M} \oplus \A^{2nk}$. There is a stabilization argument (see \cite[\S 3]{llp}) which allows to make the construction independent of the choice of $n,k$ as well. 

In order to get the product formula we use $X:=[0,1]$ for the definition of the signature class of $(M,\F_M)$. The homology of $\ra X$ is isomorphic to $\bbbc^2$, which we endow with the standard skewhermitian form. We identify $\bbbc^2 \ten \F_{\ra X}$ with $\A^{2k}$.

Now we apply the product formula to $\sigma^L(M\cup X,\F_M \cup \F_X)$. The additional assumption on $N$ implies that the de Rham operator on $\Omega^*(\ra(X\times N),\F_{\ra X}\boxtimes \F_N)$ fulfills Assumption \ref{ass} if $N$ is odd-dimensional and Assumption \ref{assodd} if $N$ is even-dimensional. This is clearly necessary for the application of the product formula. We used that $\ra(X \times N)=N \cup N$.

Before we can formulate the result we need an additional definition for $N$ even-dimensional:
Let $L \subset \cH_{\ra M} \oplus \A^{2k}$ be a Lagrangian. Then $L_{\ten} \in \cH_{\ra M\times N} \oplus \A^{2k}\ten \cH_N$. Since $\cH_N$ is projective, we may embed $\cH_N$ into $\B^j$ for $j$ large enough. Let $V$ be the orthogonal complement of $\cH_N$ in $\B^j$ and let $L_0 \in \bbbc^{2}$ be a Lagrangian. We define the Lagrangian $$\tilde L_{\ten}=L_{\ten} \oplus L_0 \ten \A^k \ten V \subset \cH_{\ra M \times N} \oplus \A^{2k} \ten \B^j \ .$$

\begin{prop}
\label{propgen}
Let $M$ be odd-dimensional and let Assumption \ref{assodd} hold for the de Rham operator on $\Omega^*(\ra M,\F_{\ra M})$.

\begin{enumerate}
\item
Let $N$ be odd-dimensional, and let Assumption \ref{ass} hold for the de Rham operator on $\Omega^*(N,\F_N)$. Then $\sigma(N,\F_N)=0$ and $\sigma(M \times N,\F_M \boxtimes \F_N)=0$. 
\item
Let $N$ be even-dimensional, and let Assumption \ref{assodd} hold for the de Rham operator on $\Omega^*(\ra M \times N,\F_{\ra M}  \boxtimes \F_N)$ and for the de Rham operator on $\Omega^*(N,\F_N)$. 
Then $$\sigma^L(M,\F_M) \ten \sigma(N,\F_N)=\sigma^{\tilde L_{\ten}}(M \times N,\F_M \boxtimes \F_N) \in K_1(\A \ten \B) \ .$$
\end{enumerate}
\end{prop}

We leave it to the reader to formulate a generalization of the proposition involving tensor products as in \S \ref{highsig}.

\begin{proof}
(1) If Assumption \ref{ass} holds, then there is a trivializing operator for $\D^{sign}_N$. Thus its index vanishes. Choose a Lagrangian $L \subset \cH_{\ra M}\oplus \A^{2k}$. We get that 
\begin{align*}
0&=\sigma^L(M\cup X,\F_M\cup \F_X) \ten \sigma(N,\F_N)\\
&=\sigma((M\cup X)\times N,(\F_M \cup \F_X)\boxtimes \F_N)\\
&=\sigma(M\times N,\F_M\boxtimes \F_N) + \sigma(X\times N,\F_X \boxtimes \F_N)\\
&=\sigma(M\times N,\F_M\boxtimes \F_N) \ .
\end{align*}
Here the second equality follows from Prop. \ref{prododod}.

(2) From the definition of $\sigma^L(M,\F_M)$ from above and Prop. \ref{prododev} we get that 
$$\sigma^L(M,\F_M) \ten \sigma(N,\F_N)=\sigma^{L_{\ten}}((M \cup X) \times N,(\F_M \cup \F_X) \boxtimes \F_N) \ .$$ 

We consider the manifold $Y:=(M \times N) \cup (X\times N) \cup X$ and the bundle 
$\F_Y:=(\F_M\boxtimes \F_N) \cup (\F_X \boxtimes \F_N) \cup (\F_X \ten \B^j)$ on $Y$.

We define Lagrangians $L_1,L_2 \subset \cH_{\ra M\times N} \oplus (\A^{2k}\ten \cH_N) \oplus (\A^{2k}\ten \B^j)$
by $$L_1=\{(x,y,z)~|~(x,y) \in L_{\ten},~ z \in L_0\ten \A^k \ten\B^j\}$$
$$L_2=\{(x,y,z)~|~(x,z) \in \tilde L_{\ten},~ y \in L_0\ten \A^k \ten\cH_N \}$$
 
It holds that 
$$\sigma^{L_1}(Y,\F_Y)=\sigma^{L_{\ten}}((M \cup X) \times N,(\F_M \cup \F_X) \boxtimes \F_N)$$
and, by definition, that
$$\sigma^{L_2}(Y,\F_Y)=\sigma^{\tilde L_{\ten}}(M\times N,\F_M\boxtimes \F_N) \ .$$

It remains to calculate $[L_1-L_2]$. 

Note that $L_2$ is constructed from $L_1$ by interchanging the last two coordinates on the subspace $\cH_{\ra M\times N} \oplus (\A^{2k}\ten \cH_N) \oplus (\A^{2k}\ten \cH_N)$. Let $U(t)$ be the unitary which equals the identity on $\cH_{\ra M\times N} \oplus (\A^{2k}\ten \cH_N) \oplus (\A^{2k}\ten V)$ and on  $\cH_{\ra M\times N} \oplus (\A^{2k}\ten \cH_N) \oplus (\A^{2k}\ten \cH_N)$ equals 
$$\left(\begin{array}{ccc} 1 & 0 & 0 \\ 0 & e^{2it}\cos(t) &\sin(t) \\ 0 & - e^{2it}\sin(t) & \cos(t) \end{array}\right) \ .$$
Then $U(t)L_1$ is a path of Lagrangians with $U(0)L_1=L_1$ and $U(\frac \pi 2)L_1=L_2$. Thus $[L_1-L_2]=0$. 
\end{proof}

\subsection{Normalization and homotopy invariance for higher signatures}
\label{normalization}
For the proof of the homotopy invariance of the higher signatures in \cite{llp} the normalization of \cite{hs} of the signature operator and the grading was used. We recall it here: In the even case define $D\alpha=i^p d\alpha$ for $\alpha \in \Omega^p(M,\Pj_M)$ and set $d_M^{sign,HS}:=D+D^*$. Define the grading operator $\tau_M^{HS}$ by $\tau_M^{HS}\alpha=i^{-p(n-p)}*\! \alpha$, where $*$ is the (standard) Hodge duality operator. (It differs from the one in \cite[Def. 3.57]{bgv}, which agrees with our $\tau_M$.) Let $U$ be the unitary defined by $U\alpha=i^{p(p-1)/2}\alpha$. Then $UdU^*=D$, hence $Ud_{M}^{sign}U^*=d_M^{sign,HS}$. Furthermore $U\tau_MU^*=(-1)^{n/2}\tau_M^{HS}$ with $n=\dim M$. The Clifford operations are also unitarily equivalent, since they are determined by signature operator. Hence for the canonical symmetric trivializing operator $A$ with respect to $\alpha_M$ we get that $[\D_M^{sign,HS}(U^*AU)]=(-1)^{n/2}[\D_M^{sign}((-1)^{n/2}A)]$. The right hand side equals $(-1)^{n/2}\sigma(M,\Pj_M)$ and the left hand side equals the  signature class defined in \cite{llp}. Thus both classes agree up to sign and in particular agree in the classical case, when the dimension is divisible by four. 

The homotopy invariance of the Chern character of $[\D_M^{sign,HS}(U^*AU)]$, and hence of the Chern character of $\sigma(M,\Pj_M)$, follows from the equality established in the Appendix of \cite{llp}. The equality was proven there under slightly stronger conditions. However it seems that the proof can be adapted as needed here. It also seems to the author that the proof already shows the equality on the level of $K$-theory classes.

We leave the consideration of the odd case to the interested reader.

\textsc{Leibniz-Arbeitsstelle Hannover\\der G\"ottinger Akademie der Wissenschaften\\
Waterloostr. 8\\
30169 Hannover\\
Germany} 

\textsc{Email: wahlcharlotte@googlemail.com}

\end{document}